\UseRawInputEncoding
\documentclass[11pt]{amsart}
\usepackage{nccmath}

\usepackage{amsmath,amsthm,amssymb}
\usepackage{times}
\usepackage{enumerate}
\usepackage{color}

\newtheorem{theorem}{Theorem}[section]

\newtheorem{definition}{Definiton}[section]

\newtheorem{lemma}{Lemma}[section]
\newtheorem{proposition}{Proposition}[section]
\newtheorem{corollary}{Corollary}[section]
\newtheorem{remark}[proposition]{Remark}

\def\e{\epsilon}

\def\n{\nabla}

\def\p{\partial}

\def\a{\alpha}

\def\n{\nabla}

\def\tr{\mathrm{tr}}

\def\p{\partial}
\def\e{\epsilon}
\def\a{\alpha}

\def\k{\kappa}
\def\l{\lambda}
\def\s{\sigma}

\def\n{\nabla}
\def\<{\langle}
\def\>{\rangle}

\def\n{D}

\def\tr{\mathrm{tr}}

\def\p{\partial}

\def\l{\lambda}
\def\s{\sigma}

\def\Rn{\mathbb R^n}

\def\lto{\left(}
\def\rto{\right)}

\def\R{\mathbb{R}}

\def\rr{\mathbb{R}}

\def\lt{\left}
\def\rt{\right}
\def\ur{u_R}
\def\a0{\alpha_0}
\def\uv{\underline{v}}
\def\tx{\tilde{x}}

\def\al{\alpha}
\def\tr{\tilde{r}}
\def\goto{\rightarrow}
\def\uer{u^{\e_R}_R}

\newcommand{\be}{\begin{equation}}
\newcommand{\ee}{\end{equation}}
\newcommand{\ju}[2]{\begin{array}{#1}#2\end{array}}

\numberwithin{equation}{section}

\begin{document}
\title[Degenerate Hessian equations on exterior domains]
{Generalized Minkowski inequality via degenerate Hessian equations on exterior domains}

%\author{Zhizhang Wang}
%\address{School of Mathematical Science, Fudan University, Shanghai, China}
%\email{zzwang@fudan.edu.cn}
\author{Ling Xiao}
\address{Department of Mathematics, University of Connecticut,
Storrs, Connecticut 06269}
\email{ling.2.xiao@uconn.edu}
\thanks{2010 Mathematics Subject Classification. Primary 35J60; Secondary 53C21, 53C24.}
%\thanks{Research of the first author is  sponsored by Natural Science  Foundation of Shanghai, No.20JC1412400,  20ZR1406600 and supported by NSFC Grants No.11871161.}

\begin{abstract}
In this paper, we prove a generalized Minkowski inequality holds for any smooth, $(k-1)$-convex, starshaped domain $\Omega.$
Our proof relies on the solvability of the degenerate $k$-Hessian equation on the exterior domain $\R^n\setminus\Omega.$
\end{abstract}

\maketitle

\section{Introduction}

In \cite{AFM}, Agostiniani-Fogagnolo-Mazzieri considered the $p$-Laplacian equation in an exterior domain  $\mathbb R^n\setminus\bar{\Omega}$,
 \begin{equation}\label{eqn:1.00}
\left\{
\begin{aligned}
&\Delta_p u={\rm div}(|\n u|^{p-2}\n u)=0\ \ {\rm in}\ \ \mathbb R^n\setminus\bar{\Omega}\\
&u=1\ \ {\rm on} \ \ \partial\Omega\\
&u(x)\rightarrow 0\ \ {\rm as}\ \ |x|\rightarrow\infty,
\end{aligned}\right.
\end{equation}
where $\Omega\subset\R^n$ is a bounded
open set with smooth boundary and $1 < p < n.$
The existence, regularity, and asymptotic behavior of \eqref{eqn:1.00} is clear:
\begin{itemize}
\item[(i)] there is a unique solution $u\in C^{1, \alpha}(\mathbb R^n\setminus \Omega)$ and $u\in C^\infty((\mathbb R^n\setminus \Omega)\setminus \{\n u\neq 0\})$.
\item[(ii)] up to some constant, $$u(x)\to \Gamma_p(x)=\frac{p-1}{n-p}\left(\frac1{\omega_{n-1}}\right)^{\frac{1}{p-1}}|x|^{\frac{p-n}{p-1}},\hbox{ as }|x|\to \infty.$$
Moreover, $$\lim_{|x|\rightarrow+\infty}\frac{u(x)}{\Gamma_{p}(x)}={{\rm C}_{p}}(\Omega)^{\frac1{p-1}},$$
where
\begin{equation*}
{\rm C}_{p}(\Omega)={\rm inf}\Big\{
\int_{\mathbb R^n}|\n v|^pdx\,\Big|\, v\in C_{c}^{\infty}(\mathbb R^n),\, v\geq1\ \ {\rm on}\ \ \Omega\Big\}.
\end{equation*}
\end{itemize}
By applying $(i)$ and $(ii),$  Agostiniani-Fogagnolo-Mazzieri proved the following $L^p$-Minkowski inequality
\begin{theorem}(Theorem 1.2 of \cite{AFM})
\label{int-thm1}
Let $\Omega\subset\R^n$
be an open bounded set with smooth
boundary. Then, for every $1 < p < n$ the following inequality holds
\be\label{int1.1}
C_p(\Omega)^{\frac{n-p-1}{n-p}}\leq\frac{1}{|\mathbb{S}^{n-1}|}\int_{\p\Omega}\left|\frac{H}{n-1}\right|^pd\sigma,
\ee
where $C_p(\Omega)$ is the normalised $p$-capacity of $\Omega.$  Moreover, equality
holds in \eqref{int1.1} if and only if $\Omega$ is a ball.
\end{theorem}

Inspired by \cite{AFM}, in this paper, we study the following Hessian equation in an exterior domain $\mathbb R^n\setminus\bar{\Omega}$,
\begin{equation}\label{k-hessian}
\left\{
\begin{aligned}
&S_k(D^2 u)=0\ \ {\rm in}\ \ \mathbb R^n\setminus\bar{\Omega}\\
&u=-1\ \ {\rm on} \ \ \partial\Omega\\
&u(x)\rightarrow 0\ \ {\rm as}\ \ |x|\rightarrow\infty,
\end{aligned}\right.
\end{equation}
where $\Omega$ is a bounded, $(k-1)$-convex, star-shaped domain and $k<\frac{n}{2}.$
\begin{remark}
It is easy to see that when $k=\frac{n}{2},$ the rotationally symmetric solution to equation
\[\left\{
\begin{aligned}
&S_k(D^2 u)=0\ \ {\rm in}\ \ \mathbb R^n\setminus\bar{B_1}\\
&u=-1\ \ {\rm on} \ \ \partial B_1\\
\end{aligned}\right.
\]
is $u=C\log|x|-1$ for any $C>0.$ We can see that as $|x|\goto\infty,$ $u(x)=C\log |x|-1\goto\infty.$ This implies that \eqref{k-hessian} is not solvable when $k=\frac{n}{2}.$ Similarly, we can show that when $k>\frac{n}{2}$, equation \eqref{k-hessian} is also not solvable. Therefore, the requirement $k<\frac{n}{2}$ is necessary.
\end{remark}
To present the result, we first introduce some notations. In this paper, $D^2u$ denotes the Hessian of $u,$ and the $k$-th elementary symmetric function
$S_k(A)$ of a symmetric matrix $A$ is defined by
\[S_k(A)=S_k(\l[A])=\sum\limits_{1\leq i_1<\cdots<i_k\leq n}\l_{i_1}\cdots\l_{i_k},\]
where $\l[A]=(\l_1, \cdots, \l_n)$ are the eigenvalues of $A.$
We also have the following definition.

\begin{definition}
For any open set $U\subset\R^n$, a function $v\in C^{1,1}(U)$ is called \textbf{k-admissible} if the
eigenvalues $\l[D^2 v(x)] =(\l_1(x), \cdots, \l_n(x))\in\bar{\Gamma}_k$ for all $x\in U,$ where
$\Gamma_k$ is the G{\aa}rding's cone
\[\Gamma_k=\{ \l\in \R^n | S_m(\l) > 0, m = 1,\cdots , k\}.\]
A $C^2$ regular hypersurface $\mathcal{M}\subset\R^{n+1}$ is called \textbf{k-convex} if its principal curvature vector
$\kappa(X)\in\Gamma_k$ for all $X\in\mathcal{M}.$
\end{definition}
We prove
\begin{theorem}
\label{int-thm2}
Let $\Omega$ be a $(k-1)$-convex, starshaped domain. Then, there exists a $k$-admissible solution $u\in C^{1,1}(\R^n\setminus\Omega)$ to equation \eqref{k-hessian},
such that $\frac{u(x)}{-|x|^{2-\frac{n}{k}}}\goto\gamma$ as $|x|\goto\infty$ in $C^2$ topology. Here, $\gamma>0$ is a constant depending on $\p\Omega$.
Moreover, for any $s\in[-1, 0),$ the level set $\{u=s\}$ is regular.
\end{theorem}

Moreover, we derive
\begin{corollary}
\label{int-cor1}
Let $\Omega$ be a $(k-1)$-convex, starshaped domain. Then we have
\be\label{int1.2}
\int_{\p\Omega}|\nabla u|^{k+\beta}\s_{k-1}(\p\Omega)dx\geq |\mathbb{S}^{n-1}|{n-1\choose k-1}\lt(\frac{n-2k}{k}\rt)^{k+\beta}\gamma^{k-\frac{\beta}{n/k-2}}.
\ee
Here, $u,$ $\gamma$ are the same as in Theorem \ref{int-thm2} and $\beta$ is an arbitrary constant satisfying $\beta\geq \frac{n-2k}{n-k}.$ In particular, when $\beta=n-2k$ we get
\be\label{int1.2*}
\int_{\p\Omega}|\nabla u|^{n-k}\s_{k-1}(\p\Omega)dx\geq |\mathbb{S}^{n-1}|{n-1\choose k-1}\lt(\frac{n-2k}{k}\rt)^{n-k}.
\ee
Moreover, equality holds in \eqref{int1.2} and \eqref{int1.2*} if and only if $\Omega$ is a ball.
\end{corollary}

The classical Minikowski inequality \cite{Min} proved in early 20 century states: If $\Omega\subset\R^3$ is a convex domain with smooth boundary, then
\[\frac{1}{|\mathbb{S}^{n-1}|}\int_{\p\Omega}\frac{H}{n-1}d\sigma\geq\lt(\frac{|\p\Omega|}{|\mathbb{S}^{n-1}|}\rt)^\frac{n-2}{n-1}.\]
Since then, the Minkowski inequality has been studied intensively  by various authors under different settings (see \cite{BHW, GL09, GMTZ, GaS, Wei} and references therein).
Note that the classical Minikowski inequality gives a lower bound for the total mean curvature of $\p\Omega.$ Here we generalize this inequality and give a lower bound for the weighted total $\s_{k-1}$ curvature of $\p\Omega.$

We now give an outline of this paper. In order to solve equation \eqref{k-hessian} and study the asymptotic behavior of the solution, we constructed the following approximating Dirichlet problem in Section \ref{tap},
\begin{equation}\label{int-ed1}
\left\{
\begin{aligned}
&S_k(D^2 u)=f_{\epsilon}\ \ {\rm in}\ \   \  B_R\setminus\bar{\Omega}\\
&u=-1\ \ {\rm on} \ \ \partial\Omega\\
&u(x)=\phi(x,C_0)\ \ {\rm on} \ \ \partial B_R.
\end{aligned}\right.
\end{equation}
Here $f_\e\goto 0$ as $\e\goto 0$ and $\Phi(x, C_0)$ is a supersolution of \ref{k-hessian}. We show the solvability of \eqref{int-ed1}
in Sections \ref{section-cs} and \ref{section sap}. The solvability of \eqref{k-hessian} then follows by standard arguments and the first part of Theorem \ref{int-thm2} is proved.

We want to point out that Section \ref{section-cs} is the most important and innovative part of this paper, where we explicitly constructed a subsolution to equation \eqref{k-hessian} in a large neighborhood of $\p\Omega.$
Note that, here $\p\Omega$ is $(k-1)$-convex, the usual distance function (that is used for constructing subsolutions) can only be defined in a very small neighborhood of $\p\Omega,$ which is not sufficient for solving this problem. Therefore, we have to utilize the additional assumption that $\p\Omega$ is starshaped, and find a replacement of the distance function. We believe that, the idea developed in this paper will be useful in solving similar problems in warped product spaces.

In Sections \ref{section de} and \ref{section ab}, we establish the desired asymptotic behavior of the solution $u$ of \eqref{k-hessian} as $|x|\goto \infty.$ This proves the second half of Theorem \ref{int-thm2}.

In Section \ref{section mq} we show for the solution $u$ of \eqref{k-hessian}, when $\beta\ge \frac{n-2k}{n-k},$
\[\Phi(\tau)=\int_{\{u=1/\tau\}}  \frac{1}{|\n u|}S_k^{ij}u_iu_j\left(\frac{|\n u|}{(-u)^{\frac{k-n}{2k-n}}}\right)^{\beta}  dx\]
is monotone. Combining the monotonicity of $\Phi(\tau)$ with Theorem \ref{int-thm2}, we obtain Corollary \ref{int-cor1}.
\section*{Acknowledgements}
The author is grateful to Zhizhang Wang for his interest and for many fruitful discussions. The author is grateful to Chao Xia for contributions of Section \ref{section mq} and for bringing this problem to her attention. Without their help, this paper would not have been possible.
The author is also grateful to Joel Spruck for carefully reading the draft and for useful comments to improve the paper.

\section{The approximate problem}
\label{tap}
Equation \eqref{k-hessian} is a degenerate equation defined on a non-compact domain. It is natural to approach it using a sequence of non-degenerate equations defined on compact domains $\bar{B}_R\setminus\Omega$. The delicacy here is that we want our approximate problem keeping the asymptotic behavior of the solution $u.$ More specifically, assume $u_R$ is a solution of the approximate problem defined on $\bar{B}_R\setminus\Omega,$ we want $u_R$ satisfying $|u_R|=O(R^{2-\frac{n}{k}})$ on $\p B_R.$
\subsection{Global barriers}
\label{gb}
Inspired by the rotationally symmetric solution to \eqref{k-hessian},
for any fixed small $\e_0>0$ and $0<\epsilon<\epsilon_0$, we let
$$\phi(x, C)=-C(|x|+\epsilon)^{2-\frac{n}{k}}.$$
Then a straightforward calculation shows
$$\phi_i=C\left(\frac{n}{k}-2\right)(|x|+\epsilon)^{1-\frac{n}{k}}\frac{x_i}{|x|},$$ and
$$\phi_{ij}=C\left(\frac{n}{k}-2\right)(|x|+\epsilon)^{-\frac{n}{k}}\left[(|x|+\epsilon)\left(\frac{\delta_{ij}}{|x|}-\frac{x_ix_j}{|x|^3}\right)
+\left(1-\frac{n}{k}\right)\frac{x_ix_j}{|x|^2}\right].$$
We can see that the eigenvalues of $D^2\phi$ are
$$C\left(\frac{n}{k}-2\right)(|x|+\epsilon)^{-\frac{n}{k}}\left(1-\frac{n}{k}, 1+\frac{\epsilon}{|x|}, 1+\frac{\epsilon}{|x|},\cdots, 1+\frac{\epsilon}{|x|}\right).$$
Therefore, we get
$$\sigma_k(D^2\phi)=\left(1+\frac{\epsilon}{|x|}\right)^{k-1}\left(\begin{matrix}n-1\\ k\end{matrix}\right)\frac{\epsilon}{|x|}\left(\frac{1}{|x|+\epsilon}\right)^n\left(\frac{n}{k}-2\right)^kC^k.$$
Denote $f_{\epsilon}=\sigma_k(D^2\phi(x, 1)),$ it is clear that when $n>2k$, we have $f_{\epsilon}>0$.
Moreover, consider
\begin{equation}\label{k-hessian-e}
\left\{
\begin{aligned}
&S_k(D^2 u)=f_{\epsilon}\ \ {\rm in}\ \ \mathbb R^n\setminus\bar{\Omega}\\
&u=-1\ \ {\rm on} \ \ \partial\Omega\\
&u(x)\rightarrow 0\ \ {\rm as}\ \ |x|\rightarrow\infty,
\end{aligned}\right.
\end{equation}
we can see that for $0<C_0<1<C_1,$ $\phi(x,C_0), \phi(x,C_1)$ are super- and sub-solutions of \eqref{k-hessian-e} respectively. Here, $C_0, C_1$ are chosen such that
$\phi(x, C_0)>-1$ on $\p\Omega$ and $\phi(x, C_1)<-1$ on $\p\Omega.$ In this paper, $C_0$ and $C_1$ are fixed constants.

\subsection{The approximate problem} By discussions in subsection 2.1, in the next two sections, we will study the solvability of the following approximate Dirichlet problem:
\begin{equation}\label{ed1}
\left\{
\begin{aligned}
&S_k(D^2 u)=f_{\epsilon}\ \ {\rm in}\ \   \  B_R\setminus\bar{\Omega}\\
&u=-1\ \ {\rm on} \ \ \partial\Omega\\
&u(x)=\phi(x,C_0)\ \ {\rm on} \ \ \partial B_R,
\end{aligned}\right.
\end{equation}
where $B_R$ is an $n$-dimensional ball with radius $R$ centered at $0$.

It turns out that the existence of the solution to equation \eqref{ed1} is highly non-trivial. The biggest challenge here is the $C^2$ boundary estimates on $\p\Omega.$ It is well known that in order to obtain $C^2$ boundary estimates for Dirichlet problems, we need to construct a subsolution of the Dirichlet problem
in a neighborhood of its boundary (see \cite{CNS3, Guan98, Guan99, GuanLi96} for example). We also know that, in literature, the most common way to construct such subsolution is to use the distance function (see \cite{GT83} for example).
However, here $\p\Omega$ is $(k-1)$-convex, the distance function can only be well defined in a small neighborhood of $\p\Omega.$
Since we do not have good control on the height of the solution $u_R$ of \eqref{ed1} near $\p\Omega,$ a distance function defined near $\p\Omega$ is not sufficient to construct a desired subsolution. In Section \ref{section-cs} we will address this problem.

\section{Construction of the subsolution of \eqref{ed1}}
\label{section-cs}
This section is the most important section in this paper. We discovered a replacement of the distance function for starshaped domain $\Omega$. This function behaves like a distance function near the boundary $\p\Omega$ and is well defined in $\R^n\setminus\Omega.$
\subsection{Hessian in the spherical coordinate}
\label{hess}
Let $f: \R^n\goto \R$ be a scalar function, then $f$ can also be expressed as a function of $(\theta, r)\in\mathbb{S}^{n-1}\times\R.$
Note that $g_E=r^2dz^2+dr^2,$ where $dz^2$ is the standard metric on $\mathbb{S}^{n-1}.$ In the following, we will denote the standard connection by $D$.
Now, we choose a local orthonormal frame $\{e_1, \cdots, e_{n-1}\}$ on the unit sphere $\mathbb{S}^{n-1}$. Let $\tau_{a}=\dfrac{e_{a}}{r}$, $1\leq a\leq n-1,$ which is the orthonormal frame on the sphere with radius $r,$ and we also let $\tau_r=\frac{\p}{\p r}$.
Then we have
\be\label{hess1}
Df=f_r\tau_r+\sum\limits_{a=1}^{n-1}\frac{1}{r}f_{a}\tau_{a},
\ee
where $f_r=\tau_r f, f_{a}=e_{a}f$,
and
\[
\begin{aligned}
D^2f&=D_{\tau_r}\lt(f_r\tau_r+\sum\limits_{a=1}^{n-1}\frac{1}{r}f_{a}\tau_{a}\rt)\otimes \tau_r
+\sum\limits_{b=1}^{n-1}\frac{1}{r}D_{e_{b}}\lt(f_r\tau_r+\sum\limits_{\alpha=1}^{n-1}\frac{1}{r}f_{a}\tau_{a}\rt)\otimes \tau_{b}.
\end{aligned}
\]
Since $r\tau_r$ is the position vector and the principal curvature of the radius $r$ sphere is $1/r$, we get for $1\leq a, b\leq n-1,$
\[D_{\tau_r}\tau_r=0,\,\, D_{\tau_{a}}(r\tau_r)=\tau_{a},\]
\[ D_{\tau_r}\tau_{a}=-\frac{1}{r^2}e_{a}+\frac{1}{r}D_{e_{a}}\tau_r=0,\,\,\mbox{and $D_{\tau_{b}}\tau_{a}=-\frac{1}{r}\tau_r\delta_{ab}.$}\]

Thus, the Hessian of $f$ is
\be\label{hess1.1}D^2_{ab}f=D^2f(\tau_{a},\tau_{b})=\frac{1}{r^2}f_{ab}+\frac{1}{r}f_r\delta_{ab},\ee
\be\label{hess1.2}D^2_{a r}f=D^2f(\tau_{a},\tau_r)=\frac{1}{r}f_{a r}-\frac{1}{r^2}f_{a},\ee
\be\label{hess1.3}D^2_{rr}f=D^2f(\tau_r,\tau_r)=f_{rr},\ee
where $1\leq a, b\leq n-1,$ $f_{ab}=e_{b}e_{a}f, f_{a r}=\tau_re_{a} f,$ and $f_{rr}=\tau_r\tau_rf$. Below, we will denote $f_n=\frac{\p f}{\p r}.$

\subsection{Construction of the subsolution}
\label{cons-sub}
Recall that we have assumed $\Gamma:=\p\Omega$ is starshaped and $(k-1)$-convex. We can parametrize $\Gamma$ as a graph of the radial function
$\rho(\theta):\mathbb{S}^{n-1}\goto\R,$ i.e., $\Gamma=\lt\{\rho(\theta)\theta\mid \theta\in\mathbb{S}^{n-1}\rt\}.$
Same as in Subsection \ref{hess}, suppose $\{e_1,\cdots, e_{n-1}\}$ is an orthonormal frame on the unit sphere. Then the second fundamental form of $\Gamma$ is
\[h_{ij}=\frac{\rho}{w}\lt(\delta_{ij}+2\frac{\rho_i\rho_j}{\rho^2}-\frac{\rho_{i,j}}{\rho}\rt),\]
where $w=\sqrt{1+\frac{|\nabla\rho|^2}{\rho^2}},$ $\rho_{i,j}=\nabla_{ij}\rho,$ and $\nabla$ denotes the Levi-Civita connection on $\mathbb{S}^{n-1}.$  Since $\nabla_{e_i}e_j=0$, we get $\nabla_{ij}\rho=e_je_i\rho$. In the following, for any function $f$ defined on the unit sphere, we denote $f_{ij}=e_je_i f,$ then we have $\rho_{i,j}=\rho_{ij}.$

We define $\varphi=\log\rho,$ it is clear that the second fundamental form of $\Gamma$ can be expressed as follows.
\[h_{ij}=\frac{\rho}{w}\lt(\delta_{ij}+\varphi_i\varphi_j-\varphi_{ij}\rt),\]
where $w=\sqrt{1+|\nabla\varphi|^2}.$ By a direct calculation, we also obtain
$$g_{ij}=\rho^2(\delta_{ij}+\varphi_i\varphi_j),\,\,
g^{ij}=\frac{1}{\rho^2}\lt(\delta_{ij}-\frac{\varphi_i\varphi_j}{w^2}\rt),\,\,\mbox{and $\gamma^{ij}=\frac{1}{\rho}\lt(\delta_{ij}-\frac{\varphi_i\varphi_j}{w(1+w)}\rt)$}.$$
Here $\gamma^{ij}$ is the square root of $g^{ij}.$
Let $a_{ij}=\gamma^{ik}h_{kl}\gamma^{lj},$ then the eigenvalues of $(a_{ij})_{1\leq i, j\leq n-1},$ denoted by $\kappa[a_{ij}]=(\kappa_1, \cdots, \kappa_{n-1})$ are the principal curvatures of $\Gamma.$

Now, at any point $p\in \mathbb{S}^{n-1},$ we may rotate the coordinate such that $|\nabla\rho|=\rho_1$ and $\rho_{\al\beta}=\rho_{\al\al}\delta_{\al\beta}$
for $2\leq\al,\beta\leq n-1.$ Then at the point $\hat{p}=\rho(p)p\in\Gamma$ we have
\[
\left\{
\begin{aligned}
\gamma^{11}&=\frac{1}{\rho}\lt(1-\frac{w^2-1}{w(1+w)}\rt)=\frac{1}{\rho w},\\
\gamma^{1\al}&=0,\,\,&2\leq\al\leq n-1,\\
\gamma^{\al\beta}&=\frac{1}{\rho}\delta_{\al\beta},\,\,&2\leq\al, \beta\leq n-1,
\end{aligned}
\right.
\]
and
\[
\left\{
\begin{aligned}
a_{11}&=\gamma^{1k}h_{kl}\gamma^{l1}=\gamma^{11}h_{11}\gamma^{11}=\frac{h_{11}}{\rho^2 w^2},\\
a_{1\al}&=\gamma^{1k}h_{kl}\gamma^{l\alpha}=\frac{h_{1\al}}{\rho^2 w},\,\,&2\leq\al\leq n-1,\\
a_{\al\beta}&=\gamma^{\al\al}h_{\al\beta}\gamma^{\beta\beta}=\frac{1}{\rho^2}h_{\al\beta},\,\,&2\leq\al, \beta\leq n-1.
\end{aligned}
\right.
\]
We consider $$g=r\rho^{-1},$$ and we will compute the Hessian of $g$ at an arbitrary point $(p, r)\in\lt(\mathbb{S}^{n-1}\times\R\rt)\setminus\{0\}.$
A straightforward calculation yields
\[g_1=e_1g=-r\rho^{-2}\rho_1,\,\,g_{1a}=e_{a}e_1 g=-r\lt(-2\rho^{-3}\rho_1\rho_{a}+\rho^{-2}\rho_{1a}\rt)\,\,\mbox{ for $1\leq a\leq n-1,$}\]
\[g_{1\al}=-r\rho^{-2}\rho_{1\al}\,\,\mbox{for $2\leq\al\leq n-1,$}\]
\[g_\al=e_{\al}g=0,\,\,g_{\al\beta}=e_{\beta}e_{\al}g=-r\rho^{-2}\rho_{\al\beta}\,\,\mbox{for $2\leq\al, \beta\leq n-1,$}\]
and
\[g_n=\frac{\p g}{\p r}=\rho^{-1},\,\, g_{nn}=\frac{\p^2 g}{\p r^2}=0,\,\, g_{1n}=\frac{\p g_1}{\p r}=-\rho^{-2}\rho_1,\,\, g_{\al n}=\frac{\p g_{\al}}{\p r}=0.\]
Following the notation in Subsection \ref{hess}, suppose $\tau_{a}=\frac{e_{a}}{r}, 1\leq a\leq n-1$ and $\tau_r=\frac{\p}{\p r},$ then $\{\tau_1, \cdots, \tau_{n-1}, \tau_r\}$ forms an orthonormal frame at $(p,r)$. By virtue of \eqref{hess1.1}, \eqref{hess1.2}, and \eqref{hess1.3} we get,
\begin{eqnarray}
D^2_{11}g&=&D^2g(\tau_1,\tau_1)=\frac{1}{r^2}g_{11}+\frac{1}{r\rho}=\frac{w^3}{r}a_{11}\label{cs2.1},\\
D^2_{1\al}g&=&D^2g(\tau_1,\tau_{\al})=\frac{1}{r^2}g_{1\al}=\frac{w^2}{r}a_{1\al},\,\,2\leq\al\leq n-1,\label{cs2.2}\\
D^2_{\al\beta}g&=&D^2g(\tau_{\al},\tau_{\beta})=\frac{1}{r^2}g_{\al\beta}+\frac{\delta_{\al\beta}}{r\rho}=\frac{w}{r}a_{\al\beta},\,\, 2\leq\al, \beta\leq n-1,\label{cs2.3}\\
D^2_{ij}g&=&D^2g(\tau_i,\tau_j)=0,\,\, 1\leq i, j\leq n \,\,\mbox{for all other cases.}\nonumber
\end{eqnarray}
We also notice that at $p$ we have
\[h_{\al\beta}=\frac{\rho}{w}\lt(\delta_{\al\beta}-\frac{\rho_{\al\beta}}{\rho}\rt)=h_{\al\al}\delta_{\al\beta},\,\,2\leq\al, \beta\leq n-1,\]
this implies $a_{\al\beta}=a_{\al\al}\delta_{\al\beta}$ is diagonalized.
Therefore, at $(p, r)$ we obtain
\[
\begin{aligned}
\text{Hessian}(g)&=\left[\ju{ccccc}{\frac{w^3}{r}a_{11}&\frac{w^2}{r}a_{12}&\cdots&\frac{w^2}{r}a_{1n-1}&0\\
\frac{w^2}{r}a_{12}&\frac{w}{r}a_{22}&\cdots&0&0\\
\vdots&\vdots&\ddots&\vdots\\
\frac{w^2}{r}a_{an-1}&0&\cdots&\frac{w}{r}a_{n-1n-1}&0\\
0&0&\cdots&0&0}\right].
\end{aligned}
\]
We want to point out that $\kappa[a_{ij}]$ are the principal curvatures of $\Gamma$ at $\hat{p}=\rho(p)p.$

Let's consider the function $\phi=\phi(g),$ then $\phi$ is also a function defined on $\lt(\mathbb{S}^{n-1}\times\R\rt)\setminus\{0\}.$
We will compute the Hessian of $\phi$ at $(p, r).$ Denote $\phi'(p, r)=\p_g\phi=A,$ $\phi''(p, r)=\p_{gg}\phi=B,$ we get at this point, for $1\leq i,j\leq n$,
$$D^2_{ij}\phi=AD^2_{ij}g+B(\tau_ig)(\tau_jg).$$
We note that $\rho_a=e_a\rho,$ it is easy to compute
 \[
\begin{aligned}
\text{Hessian}(\phi)&=\left[\ju{ccccc}{\frac{Aw^3}{r}a_{11}+B\rho^{-4}\rho_1^2&\frac{Aw^2}{r}a_{12}&\cdots&\frac{Aw^2}{r}a_{1n-1}&-B\rho^{-3}\rho_1\\
\frac{Aw^2}{r}a_{12}&\frac{Aw}{r}a_{22}&\cdots&0&0\\
\vdots&\vdots&\ddots&\vdots\\
\frac{Aw^2}{r}a_{1n-1}&0&\cdots&\frac{Aw}{r}a_{n-1n-1}&0\\
-B\rho^{-3}\rho_1&0&\cdots&0&B\rho^{-2}}\right].
\end{aligned}
\]

In the following, we will compute $\s_k(D^2\phi).$ For our convenience, in the rest of this section we will always assume $2\leq\al, \beta, \gamma\leq n-1,$
lower case letters $2\leq i, j, m, l, s\leq n,$ and $1\leq \hat{i}, \hat{j}\leq n-1.$
\[
\begin{aligned}
&\s_k(D^2\phi)=\phi_{11}\s_{k-1}(\phi_{ij})+\s_k(\phi_{ij})-\sum\limits_{s=2}^n\phi_{1s}^2\s_{k-2}(\phi_{ij}|\phi_{ss})\\
&=\phi_{11}\lt[\s_{k-1}(\phi_{\al\beta})+B\rho^{-2}\s_{k-2}(\phi_{\al\beta})\rt]+\lt[\s_k(\phi_{\al\beta})+B\rho^{-2}\s_{k-1}(\phi_{\al\beta})\rt]\\
&-\sum\limits_{\al=2}^{n-1}\phi^2_{1\al}\lt[\s_{k-2}(\phi_{\beta\gamma}|\phi_{\al\al})+B\rho^{-2}\s_{k-3}(\phi_{\beta\gamma}|\phi_{\al\al})\rt]
-\phi^2_{1n}\s_{k-2}(\phi_{\al\beta})\\
&=\lt(A\frac{w^3}{r}a_{11}+B\rho^{-4}\rho_1^2\rt)\lt[\lt(\frac{Aw}{r}\rt)^{k-1}\s_{k-1}(a_{\al\beta})
+B\rho^{-2}\lt(\frac{Aw}{r}\rt)^{k-2}\s_{k-2}(a_{\al\beta})\rt]\\
&+\lt(\frac{Aw}{r}\rt)^k\s_k(a_{\al\beta})+B\rho^{-2}\lt(\frac{Aw}{r}\rt)^{k-1}\s_{k-1}(a_{\al\beta})\\
&-\sum\limits_{\al=2}^{n-1}\lt(\frac{Aw^2}{r}\rt)^2\lt(\frac{Aw}{r}\rt)^{k-2}a^2_{1\al}\s_{k-2}(a_{\beta\gamma}|a_{\al\al})\\
&-B\rho^{-2}\sum\limits_{\al=2}^{n-1}\lt(\frac{Aw^2}{r}\rt)^2\lt(\frac{Aw}{r}\rt)^{k-3}a^2_{1\al}\s_{k-3}(a_{\beta\gamma}|a_{\al\al})\\
&-B^2\rho^{-6}\rho_1^2\lt(\frac{Aw}{r}\rt)^{k-2}\s_{k-2}(a_{\al\beta})\\
&=\frac{A^kw^{k+2}}{r^k}[\s_k(a_{\hat{i}\hat{j}})-\s_k(a_{\al\beta})]
+B\rho^{-2}\frac{A^{k-1}w^{k+1}}{r^{k-1}}\s_{k-1}(a_{\hat{i}\hat{j}})+\lt(\frac{Aw}{r}\rt)^k\s_k(a_{\al\beta})\\
&\geq-c_0\frac{A^k}{r^k}+c_1\frac{BA^{k-1}}{\rho^2r^{k-1}}.
\end{aligned}
\]
Here $c_1=\min\limits_{\hat{q}\in\Gamma}\s_{k-1}(\kappa_{\Gamma}(\hat{q}))>0$ and $c_0=c_0(|\rho|_{C^2})$ are two positive constants only depending on
$\Gamma.$
Since $(p, r)$ is an arbitrary point in $\lt(\mathbb S^{n-1}\times\mathbb R\rt)\setminus\{0\},$  consider $\phi(g)=g^N,$ the above calculation gives
\be\label{cs2.4}
\s_k(D^2\phi)\geq\frac{N^kg^{kN-k}}{r^k}\lt[-c_0+\frac{c_1(N-1)}{\rho}\rt]\,\,\mbox{in $(\mathbb S^{n-1}\times\mathbb R)\setminus\{0\}$.}
\ee
Notice that $g\geq 1$ in $\mathbb R^n\setminus\Omega,$ choosing $N=N(\Gamma)>0$ large we have
\begin{eqnarray}\label{sk}
\s_k(D^2\phi)\geq\frac{1}{r^k}\,\,\mbox{in $\R^n\setminus\bar{\Omega}$}.
\end{eqnarray}
Moreover, \be\label{sk1}
\phi=1\,\,\mbox{on $\Gamma=\p\Omega.$}
\ee
In Section \ref{section sap}, we will use $\phi$ to construct the subsolution to equation \eqref{ed1},
which is crucial for obtaining a priori estimates that are needed for solving equation \eqref{ed1}.

\section{Solvability of the approximate problem}
\label{section sap}
Let us consider the following Dirichlet problem (i.e., equation \eqref{ed1})
\begin{equation}\label{ed1*}
\left\{
\begin{aligned}
&S_k(D^2 u)=f_{\epsilon}\ \ {\rm in}\ \   \  B_R\setminus\bar{\Omega}\\
&u=-1\ \ {\rm on} \ \ \partial\Omega\\
&u(x)=\phi(x,C_0)\ \ {\rm on} \ \ \partial B_R,
\end{aligned}\right.
\end{equation}
where $f_\e=\left(1+\frac{\epsilon}{|x|}\right)^{k-1}\left(\begin{matrix}n-1\\ k\end{matrix}\right)\frac{\epsilon}{|x|}\left(\frac{1}{|x|+\epsilon}\right)^n\left(\frac{n}{k}-2\right)^k,$ $B_R$ is an $n$-dimensional ball with radius $R$ centered at $0,$ and
$\Omega\subset\R^n$ is a $(k-1)$-convex, starshaped domain. For our convenience,
in the rest of this paper, we will always denote $\mathbf{\al_0:=\frac{n}{k}-2}>0.$
In this section, we prove
\begin{theorem}
\label{sap-thm1}
Given $\e_0>0$ small and $R_0>0$ large, then for any $0<\e<\e_0,$ $R>R_0$ there is a unique, $k$-convex solution $u_R^\e\in C^2(B_R\setminus\bar{\Omega})\cap C^{1, 1}(\bar{B}_R\setminus\Omega)$ satisfying \eqref{ed1*}.
\end{theorem}

In the following, when there is no confusion, we will drop the superscript $\e$ and write $u_R$ instead of $u_R^\e.$ We also want to point out that, in this section we will establish some fairly accurate $C^1$ and $C^2$ estimates for $u_R$ on $\p B_R.$ Later, we will see that we need these accurate estimates for studying the asymptotic behavior of the solution $u$ of \eqref{k-hessian}.

Applying the maximum principle, the $C^0$ estimate of $u_R$ follows from the existence of the global barriers constructed in Subsection \ref{gb}.
\begin{lemma}
\label{C0-lem}
(\textbf{$C^0$ bounds of $u_R$})
Let $\ur$ be the solution of \eqref{ed1*}, then $\ur$ satisfies
\[\phi(x, C_1)<\ur<\phi(x, C_0),\,\,\mbox{in $B_R\setminus\bar{\Omega}.$}\]
\end{lemma}
\subsection{$C^1$- estimates}
\label{c1-dirichlet}
\begin{lemma} (\textbf{$C^1$ upper bound on $\p\Omega$})
\label{C1-inside-upper-lem}
Let $\ur$ be the solution of \eqref{ed1*}, then on the boundary $\p\Omega$, $\ur$ satisfies
\[\frac{\p\ur}{\p\nu}<C,\]
where $\nu$ is the outward unit normal of $\p\Omega$ (pointing into $B_R\setminus\bar\Omega$), and $C=C(\p\Omega, n, k)$ only depends on $\p\Omega,$ $n,$ and $k.$
\end{lemma}
\begin{proof}
Since $\p\Omega$ is smooth, we know there exists $r_0>0$ such that for any $\xi\in\p\Omega$ there is $z_\xi\in\Omega$
satisfying $\bar{B}_{r_0}(z_\xi)\cap\p\Omega=\xi,$ where $B_{r_0}(z_{\xi})$ is a ball centered at $z_{\xi}$ with radius $r_0$. Namely, the sphere $\p B_{r_0}(z_\xi)$ and $\p\Omega$ are tangent  at $\xi$. We also denote
\[U_{\delta_0}:=\lt\{x\in\Rn\setminus \bar{\Omega}\mid \text{dist}(x, \Omega)<\delta_0\rt\}.\]
Let $c_1:=\max\limits_{x\in\p U_{\delta_0}}\phi(x, C_0),$ then consider
\[\bar{u}_\xi=-C|x-z_\xi|^{-\a0}+Cr_0^{-\a0}-1.\] It's clear that
$\bar{u}_\xi(\xi)=-1$ and $\bar{u}_\xi(x)>-1$ for any $x\in\p\Omega\setminus\{\xi\}.$
On $\p U_{\delta_0}$ we have
\[\bar{u}_\xi\geq\frac{C}{r_0^{\alpha_0}}\lt[1-\lt(1+\frac{\delta_0}{r_0}\rt)^{-\a0}\rt]-1.\]
Choosing $C=C(\a0, \delta_0,\p\Omega)=\dfrac{(1+c_1)r_0^{\al_0}}{1-\lt(1+\delta_0/r_0\rt)^{-\al_0}},$ we get $\bar{u}_\xi\geq c_1$ on $\p U_{\delta_0}.$
Since $S_k(\nabla^2\bar{u}_\xi)=0,$ by the maximum principle we obtain $\bar{u}_\xi>\ur$ in $U_{\delta_0}\setminus\bar{\Omega}.$
This implies at $\xi$ we have
\[\frac{\p\ur}{\p\nu}<\frac{\p\bar{u}_\xi}{\p\nu}=C\al_0r_0^{-\al_0-1}=\frac{(1+c_1)\al_0r_0^{-1}}{1-\lt(1+\delta_0/r_0\rt)^{-\al_0}}.\]
Since $\xi\in\p\Omega$ is arbitrary, we prove this lemma.
\end{proof}

\begin{lemma}(\textbf{$C^1$ lower bound on $\p\Omega$})
\label{C1-inside-lower-lem}
Let $0<\e<\e_0$ and $R>(2C_1)^{\frac{1}{\al_0}}$ in equation \eqref{ed1*}, where $\e_0>0$ is a small constant depending on $\p\Omega$ and $C_1$ is the fixed constant defined in
Subsection \ref{gb}. Let $\ur$ be the solution of \eqref{ed1*}. Then on the boundary $\p\Omega$, $\ur$ satisfies
\[\frac{\p\ur}{\p\nu}>C,\]
where $\nu$ is the outward unit normal of $\p\Omega$ (pointing into $B_R\setminus\bar\Omega$), and $C=C(\p\Omega,n, k)$ only depends on $\p\Omega,$ $n,$ and $k.$
\end{lemma}
\begin{proof}
Let $r_1=(2C_1)^{\frac{1}{\a0}}<R,$ then on $\p B_{r_1}$ we get
\[\ur>\phi(x, C_1)\geq-\frac{1}{2},\]
where $B_{r_1}$ is the ball of radius $r_1$ centered at the origin. We will denote $\rho_0:=\min\limits_{\theta\in\mathbb{S}^{n-1}}\rho(\theta).$
Now, let $\underline{v}=\frac{1}{A}\lt(\phi(g)-1\rt)-1,$ where $\phi(g)$ is the subsolution constructed in Subsection \ref{cons-sub}, and $A>0$ is a constant chosen to be sufficiently large such that
$\frac{1}{A}\lt[\phi(r_1/\rho_0)-1\rt]<\frac{1}{2}.$

By virtue of \eqref{sk} we know that,
$S_k(D^2\uv)\geq\frac{1}{A^kr^k}$ in $\R^n\setminus\bar{\Omega}.$ It is easy to see that when $0<\e<\e_0$ and $\e_0$ is small, we have
$S_k(D^2\uv)>f_\e$ in $B_{r_1}\setminus\bar{\Omega}.$ Moreover, $\uv$ satisfies $\uv=-1$ on $\p\Omega$ and
$\uv<\ur$ on $\p B_{r_1}.$ Therefore, the standard maximum principle yields $\uv<\ur$ in $B_{r_1}\setminus\bar{\Omega}.$
This implies for any $\xi\in\p\Omega$ we have
\[\frac{\p\ur}{\p\nu}>\frac{\p\uv}{\p\nu}.\]
Therefore, we complete the proof of Lemma \ref{C1-inside-lower-lem}.
\end{proof}

\begin{remark}\label{rmk3.1} By virtue of \eqref{hess1} we get $|D\uv|>c>0$ in $B_{r_1}\setminus\Omega.$ We also notice that $\uv>-1$ in $B_{r_1}\setminus\bar{\Omega}$ and $\uv=-1$ on $\p\Omega.$ Therefore, we obtain
$\dfrac{\p\ur}{\p\nu}>\dfrac{\p\uv}{\p\nu}>c$ on $\p\Omega.$ Here $c=c(\p\Omega)>0$ is a positive constant independent of $R$ and $\e$.

\end{remark}

\begin{lemma}(\textbf{$C^1$ upper and lower bounds on $\p B_R$})
\label{C1-outside-lem}
Let $\ur$ be the solution of \eqref{ed1*} for $R>R_0$, then on the boundary $\p B_R$, $\ur$ satisfies
\[\a0C_0(R+\e)^{-\a0-1}<|D\ur|<\a0C_1(R+\e)^{-\a0-1}.\]
Here, $R_0>0$ is a constant depending on $\phi(x, C_1)$ and $C_0, C_1$ are constants defined in Subsection \ref{gb}.
\end{lemma}
\begin{proof}
Let $\underline{u}=\phi(x, C_1)+C_1(R+\e)^{-\a0}-C_0(R+\e)^{-\a0}.$ It is clear that on $\p B_R$, $\underline{u}=u_R$. When $R>R_0>0$ is sufficiently large, we have $\underline{u}<-1$ on $\p\Omega.$ Since $\sigma_k(D^2\phi(x,C_1))>f_\e$,  we obtain $\underline{u}<\ur<\phi(x, C_0)$ in $B_R\setminus\bar{\Omega}.$ This implies on $\p B_R$ we have
\[\frac{\p\phi(x,C_0)}{\p\nu}<\frac{\p\ur}{\p\nu}<\frac{\p\underline{u}}{\p\nu},\] where $\nu$ is the outer norma to $\p B_R$ (pointing away from $B_R\setminus\bar\Omega$).
Therefore, we conclude on $\p B_R$
\[\a0 C_0(R+\e)^{-\a0-1}<|D\ur|<\a0 C_1(R+\e)^{-\a0-1}.\]
 \end{proof}

\begin{lemma}(\textbf{$C^1$ global on $B_R\setminus\bar{\Omega}$})
\label{C1-global-lem}
Let $\ur$ be the solution of \eqref{ed1*}, then there exists some constant $A>0$ independent of $R$ and $\e$ such that
\[\max\limits_{x\in\bar{B}_R\setminus\Omega}\{|D\ur|+A\ur\}=\max\limits_{x\in\p\bar{B}_R\cup\p\Omega}\{|D\ur|+A\ur\}.\]
\end{lemma}
\begin{proof}
Consider $W:=\max\limits_{x\in\bar{B}_R\setminus\Omega, \xi\in\mathbb{S}^{n}}\{D_\xi\ur+A\ur\}.$ Suppose $W$ is achieved at an interior point
$(x_0, \xi_0)\in\lto B_R\setminus\bar{\Omega}\rto\times\mathbb{S}^n.$ Rotating the coordinate we may assume at $x_0,$
$$D_{\xi_0}\ur=(\ur)_1,\,\,\mbox{and $(u_R)_{\alpha\beta}=\lambda_\alpha\delta_{\alpha\beta}$ for $\alpha,\,\, \beta\geq 2.$}$$ Then at $x_0,$ a straightforward calculation gives
$A(u_R)_i+u_{1i}=0,$ which implies
$$(u_R)_{1i}=0\,\, \mbox{for $i\geq 2,$ and $(u_R)_{11}=-A(u_R)_1.$}$$
Moreover, at $x_0$, we also have
\begin{equation}
\label{ed2}
\begin{aligned}
0&\geq AS_k^{ii}(u_R)_{ii}+S^{ii}_k(u_R)_{1ii}\\
&=Akf_\e+f_\e\frac{x_1}{|x_0|}\lt[-\frac{(k-1)\e |x_0|^{-2}}{1+\e |x_0|^{-1}}-|x_0|^{-1}-\frac{n}{|x_0|+\e}\rt],
\end{aligned}
\end{equation}
where we have used $S_k^{ii}(\ur)_{ii}=kf_\e$ and $S^{ii}_k(\ur)_{1ii}=(f_\e)_1.$
It's clear that when $A=A(n, k, \p\Omega)>0$ large, the right hand side of \eqref{ed2} is positive. This leads to a contradiction. Thus $W$ is achieved on the boundaries.
\end{proof}
Combining Lemma \ref{C0-lem}, \ref{C1-inside-upper-lem}, \ref{C1-inside-lower-lem},  \ref{C1-outside-lem} with Lemma \ref{C1-global-lem}, we conclude
\begin{proposition}\label{C1-bound}
Let $\ur$ be the solution of \eqref{ed1*}, then there exists some constant $C$ independent of $R$ and $\e$ such that
\[|D\ur|\leq C\,\,\mbox{in $\bar{B}_R\setminus\Omega.$}\]
\end{proposition}

\subsection{$C^2$ boundary estimates}
\label{c2-boundary}
Let $x_0\in\p\Omega$ be an arbitrary point on $\p\Omega.$ Without loss of generality, we may choose a local coordinate $\{\tx_1, \cdots, \tx_n\}$ such that $x_0$ is the origin. Let $\tx_n$ axis be the exterior normal of $\p\Omega$ and the boundary near the origin is represented by
\[\tx_n=\rho(\tx')=-\frac{1}{2}\sum_{\alpha=1}^n\k_\al\tx^2_\al+O(|\tx'|^3),\]
where $\k_1,\k_2,\cdots,\k_{n-1}$ are the principal curvatures of $\p\Omega$ at the origin and $\tx'=(\tx_1,\tx_2,\cdots,\tx_{n-1})$. Let $\ur$ be the solution of \eqref{ed1*}, then \be\label{ed2*}
u_R(\tx', \rho(\tx'))=-1\,\,\mbox {on $\p\Omega.$}
\ee
First, differentiating \eqref{ed2*} with respect to tangential directions at $x_0$ we obtain,
\[(u_R)_{\al}+(u_R)_{n}\rho_{\al}=0\,\,\mbox{and $(u_R)_{\al\beta}+(u_R)_{n}\rho_{\alpha\beta}=0$}.\]
This gives,
\be\label{ed3*}
(u_R)_{\alpha\beta}=(u_R)_{n}\k_{\alpha}\delta_{\al\beta}\,\,\mbox{for $\al,\beta<n.$}
\ee
Next, we estimate $|(u_R)_{\al n}(0)|.$
\begin{lemma}(\textbf{$C^2$ bound on $\p\Omega$ in mixed directions})
\label{C2-inside-mix-lem}
Let $u_R$ be the solution of \eqref{ed1*}, then on $\p\Omega,$ $u_R$ satisfies
\[|(u_R)_{\tau\nu}|<C,\]
where $\tau$ is an arbitrary unit tangential vector of $\partial\Omega$, $\nu$ is the outward unit normal of $\p\Omega$ (pointing into $B_R\setminus\bar\Omega$), and $C=C(\p\Omega, n, k)$ is a constant only depending on $\p\Omega,$ $n,$ and $k.$
\end{lemma}
\begin{proof}
In this proof, let $\underline{v}$ be the subsolution of \eqref{ed1*} constructed in the proof of Lemma \ref{C1-inside-lower-lem}. Let $x_0\in\p\Omega$ be an arbitrary point on $\p\Omega.$ Without loss of generality, we may choose a local coordinate $\{\tx_1, \cdots, \tx_n\}$ such that $x_0$ is the origin.
Let $T=\p_{\al}-\k_{\al}(\tx_{\al}\p_{n}-\tx_n\p_{\al}),$ then on $\p\Omega$ near $x_0$ we have
\[Tu_R=\lto\p_{\al}+\rho_{\al}\p_{n}\rto u_R+O(|\tx'|^2)=O(|\tx'|^2).\]
Denote $\mathcal{L}:= S_k^{ij}\p_{ij}$ and $\tilde{B}_{\delta_0}=B_{\delta_0}(x_0)\setminus\bar{\Omega},$ where $\delta_0>0$ is a fixed constant chosen to be so small that $\tilde{B}_{\delta_0}\subset B_{r_1}\setminus\bar{\Omega}.$ Note that $r_1=(2C_1)^{\frac{1}{\a0}}$ is a fixed constant defined in Lemma \ref{C1-inside-lower-lem} and $B_{r_1}$ is the ball centered at $0$ (under the original coordinate, not $x_0$) with radius $r_1$. A straightforward calculation yields
$$|\mathcal{L}Tu_R|<c_0f_\e\,\,\mbox{in $\tilde B_{\delta_0}$}$$
for some $c_0>0$ independent of $R$ and $\e.$ Notice that $\uv$ is strictly $k$-convex, i.e., $\l[D^2\underline{v}]\in\Gamma_k.$ Moreover, without loss of generality we may also assume $S_k(D^2\uv)>4f_\e$ in $B_{r_1}\setminus\bar{\Omega}.$  It is easy to see that there exists some constant
$\theta=\theta(\p\Omega)>0$ satisfying
\[\l[D^2(\underline{v}-\theta|\tx|^2)]\in\Gamma_k\] and
\[S_k[D^2(\underline{v}-\theta|\tx|^2)]>2f_{\e}\,\,\mbox{in $B_{r_1}\setminus\bar{\Omega}.$}\]

Consider the barrier $h=u_R-\underline{v}+\theta|\tx|^2,$ then it is easy to see that
$$h=\theta|\tx|^2\,\,\mbox{on $\p\tilde{B}_{\delta_0}\cap\p\Omega,$ and $h\geq \theta \delta_0^2$ on $\p\tilde{B}_{\delta_0}\setminus\p\Omega.$}$$
Moreover, by the concavity of $S_k^{1/k}$, we have $$\mathcal{L}h=\mathcal{L}u_R-\mathcal{L}(\underline{v}-\theta|\tx|^2)<(1-2^{1/k})kf_\e:=-c_1f_{\e},$$
where $c_1=k(2^{1/k}-1)$.
Now let $\omega=Tu_R+Bh,$ where the constant $B$ is chosen such that
$$c_1B>c_0,\,\,\mbox{ $B\theta|\tx|^2>|Tu_R|$ on $\p\Omega\cap\tilde{B}_{\delta_0},$
and $B\theta \delta_0^2>|T\ur|$ on $\p\tilde{B}_{\delta_0}\setminus\p\Omega.$}$$ Then we obtain
\[\mathcal{L}\omega<0\,\,\mbox{in $\tilde{B}_{\delta_0},$ and $\omega\geq 0$ on $\p\tilde{B}_{\delta_0}.$}\]
Therefore, we conclude $Tu_R>-Bh$ in $\tilde{B}_{\delta_0}.$ This implies $(Tu_R)_{n}>-Bh_{n}$ at $x_0.$
Similarly, by considering $Tu_R-Bh$ we get $(Tu_R)_{n}<Bh_{n}$ at $x_0.$
Hence, our lemma is proved.
\end{proof}

Finally, we estimate $|(u_R)_{nn}(0)|.$ Let $\k=(\k_1,\cdots,\k_{n-1})$ be the principal curvatures of $\p\Omega.$ In view of Remark \ref{rmk3.1}, we know on $\p\Omega,$ $D_{\nu}u_R>0$. Then by $S_k(D^2u_R)=f_{\e}$ and equation \eqref{ed3*}, we have
\[|D_\nu u_R|^kS_k(\k)+|D_\nu u_R|^{k-1}S_{k-1}(\k)(u_R)_{\nu\nu}-\sum_{\al=1}^{n-1}|D_\nu u_R|^{k-2}S_{k-2}(\k|\al)(u_R)^2_{\al \nu}=f_\e.\]
Combining the above equality with Lemma \ref{C1-inside-upper-lem} and \ref{C2-inside-mix-lem}, we conclude
\begin{lemma}(\textbf{$C^2$ bound on $\p\Omega$ in double normal directions})
\label{C2-inside-normal-lem}
Let $u_R$ be the solution of \eqref{ed1*}, then on $\p\Omega,$ $u$ satisfies
\[|(u_R)_{\nu\nu}|<C,\]
where $\nu$ is the outward unit normal of $\p\Omega$ (pointing into $B_R\setminus\bar\Omega$), and the constant $C=C(\p\Omega, n, k)$ only depends on $\p\Omega,$ $n,$ and $k.$
\end{lemma}

In the rest of this subsection, we will consider the $C^2$ boundary estimates on $\p B_R.$

Let $x_0\in\p B_R$ be an arbitrary point on $\p B_R.$ Without loss of generality, we may assume
$x_0=(0, \cdots, 0, R).$ Choose a local coordinate $\{\tx_1, \cdots, \tx_n\}$ around $x_0$ such that $x_0$ is the origin, and
the $\tx_n$ axis is the interior normal of $\p B_R$ (pointing into $B_R\setminus\bar\Omega$). Then the half sphere containing $x_0,$ which will be denoted by $\p B_{R+},$ can be written as
\[\tx_n=\rho(\tx')=R-\sqrt{R^2-|\tx'|^2},\,\,|\tx'|<R,\]
where $\tx'=(\tx_1,\cdots,\tx_{n-1})$. Note that on $\p B_R,$ we have $u_R(\tx', \rho(\tx'))=b_R$,
where $b_R=-C_0(R+\epsilon)^{-\al_0}$.
Therefore, on $\p B_{R+}$ for $\al, \beta<n,$ we have
\[(u_R)_{\al}+(u_R)_{n}\rho_{\al}=0,\]
and
\[(\ur)_{\al\beta}+(\ur)_{\al n}\rho_\beta+(\ur)_{\beta n}\rho_{\al}+(\ur)_n\rho_{\al\beta}=0.\]
In particular, at $x_0$ we get
\be\label{ed4*}
(u_R)_{\al\beta}=-(u_R)_{n}\rho_{\al\beta}=-\frac{(u_R)_{n}}{R}\delta_{\al\beta}.
\ee
Combining Lemma \ref{C1-outside-lem} with \eqref{ed4*} we conclude,
\be\label{ed5*}
\a0 C_0(R+\e)^{-\a0-1}R^{-1}\delta_{\al\beta}<|(u_R)_{\al\al}(x_0)|<\a0 C_1(R+\e)^{-\a0-1}R^{-1}\delta_{\al\beta}
\ee
for $\al, \beta<n.$
\begin{lemma}(\textbf{$C^2$ bound on $\p B_R$ in mixed directions})
\label{C2-outside-mix-lem}
Let $u_R$ be the solution of \eqref{ed1*}, then on $\p B_R,$ $u$ satisfies
\[|(u_R)_{\tau\nu}|<C R^{-\a0-1},\]
where $\tau$ is an arbitrary unit tangential vector of $\partial\Omega$, $\nu$ is the outward unit normal of $\p B_R$ (pointing away from $B_R\setminus\bar\Omega$), and the constant $C=C(\p\Omega, n, k)$ only depends on $\p\Omega,$ $n,$ and $k.$
\end{lemma}
\begin{proof}
We use the coordinate $\{\tx_1,\cdots,\tx_n\}$ chosen above.
Let $$Tu_R=(u_R)_{\al}+\frac{1}{R}(\tx_{\al}(u_R)_{n}-\tx_n (u_R)_{\al}),$$ then it is clear that on $\p B_{R+}$
we have $Tu_R=0.$ Denote
$$Q_R=\lt\{x\in B_R\setminus\bar{\Omega}:|\tx'|<\frac{R}{2}, \rho(\tx')<\tx_n<\lt(1-\frac{\sqrt{3}}{2}\rt)R\rt\}.$$
By Lemma \ref{C1-global-lem} we can choose $A>0$ such that $|Du_R|+Au_R$ achieves its maximum on $\p B_R\cup \p\Omega.$ Moreover, in view of the Dirichlet boundary condition, we can see that if $A,R$ are sufficiently large, then we get
$$|Du_R|\big|_{\p\Omega}-A\leq |Du_R|\big|_{\p B_R}+Ab_R.$$
Thus, the maximum of $|Du_R|+Au_R$ is achieved on $\p B_R.$
By Lemma \ref{C1-outside-lem} we derive
\[|Tu_R|\leq 3\lt[A(b_R-u_R)+\a0C_1(R+\e)^{-\a0-1}\rt]\,\,\mbox{in $\bar{Q}_R.$}\]
Since when $R>R_0$ large,
\[\underline{u}=\phi(x, C_1)+C_1(R+\e)^{-\a0}+b_R\]
is a subsolution of \eqref{ed1*}, we have in $\bar{Q}_R$
\[|Tu_R|\leq 3[A(b_R-\underline{u})+\a0C_1(R+\e)^{-\a0-1}].\]
Let $h=3A(b_R-\underline{u})+C_2(R+\e)^{-\a0-1}\sin\lt(\frac{\tx_n}{R}\rt),$ then we can choose $C_2$ so large that
$\omega>|Tu_R|$ on $\p Q_R.$ Notice that the choice of $C_2$ is independent of $R$ and $\e.$ Moreover, denote $\mathcal{L}=S^{ij}_{k}\p_{ij}$, then it is easy to see that $$\mathcal{L}h<-3Akf_\e,$$ and
\[|\mathcal{L}Tu_R|=\lt|(f_\e)_{\al}+\frac{\tx_\al}{R}(f_\e)_{n}-\frac{\tx_n}{R}(f_\e)_{\al}\rt|
\leq cf_\e\,\,\mbox{in $Q_R$,}\]
where $c>0$ is a constant independent of $R$ and $\e.$
Therefore, when $R>R_0$ and $A>0$ large we have $|\mathcal{L}Tu_R|<|\mathcal{L}h|.$ Following the argument in the proof of Lemma \ref{C2-inside-mix-lem}, we obtain $|(Tu_R)_{n}|<h_{n}$
at $x_0.$ The proof of this lemma is completed.
\end{proof}

Suppose $\k=(\k_1,\cdots,\k_{n-1})$ is the principal curvature vector of $\p B_R$. On $\p B_R,$ by Lemma \ref{C1-outside-lem} we know $|D_{\nu}u|\sim R^{-\al_0-1}.$ Moreover, it is clear that on $\p B_R$ we have $\k=O(R^{-1}), f_{\e}=O(R^{-(n+1)}),$ and
\[|D_\nu u_R|^kS_k(\k)+|D_{\nu} u_R|^{k-1}S_{k-1}(\k)(\ur)_{\nu\nu}-\sum_{\al=1}^{n-1}|D_\nu \ur|^{k-2}S_{k-2}(\k|\al)(\ur)^2_{\al \nu}=f_\e.\]
Applying Lemma \ref{C2-outside-mix-lem} we conclude the following lemma:
\begin{lemma}(\textbf{$C^2$ bound on $\p B_R$ in double normal directions})
\label{C2-outside-normal-lem}
Let $u_R$ be the solution of \eqref{ed1*}, then on $\p\Omega,$ $u$ satisfies
\[|(\ur)_{\nu\nu}|<C R^{-\a0},\]
where $\nu$ is the outward unit normal of $\p\Omega$  and the constant $C=C(\p\Omega, n, k)$ only depends on $\p\Omega,$ $n,$ and $k.$
\end{lemma}

\subsection{$C^2$ global estimates}
Differentiating $f_\e$, we get
\begin{equation}
\label{c2g1}
\lto f_\e\rto_i=f_\e\frac{x_i}{r}\lt(-r^{-1}-\frac{n}{r+\e}-\frac{(k-1)\e}{r^2+\e r}\rt),
\end{equation}
and
\begin{equation}
\label{c2g2}
\begin{aligned}
\lto f_\e\rto_{ii}&=\lto f_\e\rto_i\frac{x_i}{r}[-r^{-1}-n(r+\e)^{-1}-(k-1)\e(r^2+\e r)^{-1}]\\
&+f_\e\lt(\frac{1}{r}-\frac{x_i^2}{r^3}\rt)[-r^{-1}-n(r+\e)^{-1}-(k-1)\e(r^2+\e r)^{-1}]\\
&+f_\e\frac{x_i^2}{r^2}[r^{-2}+n(r+\e)^{-2}+(k-1)\e(r^2+\e r)^{-2}(2r+\e)].\\
\end{aligned}
\end{equation}
It is easy to see that $f_\e$ satisfies
\begin{eqnarray}\label{Cond}
|Df_\e|\leq Af_\e,\,\,\mbox{$\Delta f_\e\geq-Af_{\e}$ in $B_R\setminus\bar{\Omega}$},
\end{eqnarray}
where $A>0$ is a constant independent of $R$ and $\e.$

\begin{lemma} (\textbf{$C^2$ global bound})
\label{C2-global-bound-lem}
Let $u_R$ be the solution of \eqref{ed1*}, then $u_R$ satisfies
$$|D^2\ur|<C(1+\sup\limits_{\p B_R\cup\p\Omega}|D^2\ur|),$$
where $C=C(n, k, \p\Omega)$ is a positive constant depending on $\p\Omega,$ $n,$ $k$.
\end{lemma}
\begin{proof}
The proof here is the same as the proof in \cite{JW1}. For the completeness, we include it here. In this proof, we will drop the subscript $R$ and write $u$ instead of $\ur.$ We consider the following test function
\[
H=\Delta u+\frac{1}{2}|Du|^2.
\]
Differentiating the equation \eqref{ed1*} twice, we get
\be\label{GC2-1*}
S_{k}^{ij}u_{ijm}=(f_\e)_m,
\ee
and
\begin{equation}
\label{GC2-11}
S_k^{ij}(\Delta u)_{ij} + \sum_i S_k^{pq,rs}u_{pqi}u_{rsi} = \Delta f_{\e}.
\end{equation}
This gives
\begin{equation}
\label{GC2-12}
S_k^{ij} H_{ij} = -\sum_i S_k^{pq,rs}u_{pqi}u_{rsi} + \sum_mS_k^{ij}u_{mi}u_{mj}+\sum_mu_m(f_{\e})_m  + \Delta f_{\e}.\nonumber
\end{equation}
Without loss of generality, we assume that $H$ attains its maximum at an interior point $x_0 \in B_R\setminus\bar{\Omega}.$
Then at $x_0$ we have
\begin{equation}
\label{GC2-14}
 H_i = (\Delta u)_i +  \sum_mu_mu_{mi}=0,
\end{equation}
and
\begin{eqnarray}\label{A1}
0 \geq  - \sum_i S_k^{pq,rs} u_{pqi}u_{rsi} +  c_0f_{\e}\Delta u +\langle Du,Df_{\e}\rangle+ \Delta f_{\e}.
\end{eqnarray}
Here we have used the Newton-Maclaurin inequality to get $\sum\limits_m S_k^{ij}u_{mi}u_{mj}\geq c_0f_\e\Delta u,$ for some constant $c_0$ only depends on $n,k$.
Using the concavity of $S_k^{1/k}$, we have
\begin{eqnarray}\label{A2}
- \sum_i S_k^{pq,rs} u_{pqi}u_{rsi} \geq -\frac{|Df_{\e}|^2}{f_{\e}}.
\end{eqnarray}
Combining \eqref{Cond}, \eqref{A1}, \eqref{A2}, and Proposition \ref{C1-bound}, we get
\begin{eqnarray}
0 &\geq&  -\frac{|Df_{\e}|^2}{f_{\e}}+  c_0f_{\e}\Delta u +\langle Du,Df_{\e}\rangle+ \Delta f_{\e}\\
&\geq &-A^2f_{\e}+c_0f_{\e}\Delta u-ACf_{\e}-Af_{\e}
\end{eqnarray}
where $C$ is some constant only depending on $\p\Omega, n,$ and $k.$ Thus, we can see that if $H$ obtains its maximum at an interior point, then $H$ is bounded by a constant $C=C(n, k, \p\Omega).$ This implies the desired result.
\end{proof}
Combing  Lemmas \ref{C2-inside-mix-lem}--\ref{C2-global-bound-lem}, we obtain the uniform $C^2$ bound for $u_R$ on $\bar{B}_R\setminus\Omega.$
\begin{proposition}
\label{C2-bound}
Let $\ur$ be the solution of \eqref{ed1*}, then there exists some constant $C$ independent of $R$ and $\e$ such that
\[|D^2\ur|\leq C,\,\,\mbox{in $\bar{B}_R\setminus\Omega.$}\]
\end{proposition}

The solvability of the approximate Dirichlet problem \eqref{ed1*} follows from Lemma \ref{C0-lem}, Proposition \ref{C1-bound}, and Proposition \ref{C2-bound}. Therefore, we proved Theorem \ref{sap-thm1}.

Now, let $\lt\{u_{R_i}^{\e_i}\rt\}_{i=1}^\infty$ be a sequence of solutions to \eqref{ed1*}. We also assume as $i\goto\infty,$ $R_i\goto\infty$ and $\e_i\goto 0.$ Notice that the $C^0-C^2$ estimates we obtained in this section are independent of $R$ and $\e.$ By the standard convergence theorem, we conclude that there exists a subsequence of $\lt\{u_{R_i}^{\e_i}\rt\}_{i=1}^\infty$ converges to a function $u$ which satisfies equation \eqref{k-hessian}. This completes the proof of the first part of Theorem \ref{int-thm2}.

\section{The decay estimates}
\label{section de}
For the prototype function $|x|^{-\al_0}$, its $C^0, C^1,$ and $C^2$ decay orders are $\al_0,\al_0+1,\al_0+2$. In this section, we would like to prove that $\ur$ the solution of \eqref{ed1*} has the same decay order as $|x|^{-\al_0}$.

In view of Lemma \ref{C0-lem}, the $C^0$ decay estimate follows directly.
\begin{lemma}(\textbf{$C^0$ decay estimates})
\label{C0-interior-lem}
For any $x_0\in B_R\setminus\Omega,$ denote $|x_0|=r_0.$ Let $\ur$ be a solution to equation \eqref{ed1*}. Then we have
\begin{equation}\label{c1inside}
|\ur(x_0)|\leq B|x_0|^{-\a0}\end{equation}
for some constant $B>0$ independent of $r_0, \e,$ and $R.$
\end{lemma}

\begin{lemma}(\textbf{$C^1$ decay estimates})
\label{C1-interior-lem}
For any $x_0\in\mathbb{R}^n\setminus\Omega,$ denote $|x_0|=r_0.$ Let $\ur$ be a solution to equation \eqref{ed1*} with $R>10r_0.$ Then we have
\begin{equation}\label{c1inside}
|D\ur(x_0)|\leq B|x_0|^{-\a0-1}\end{equation}
for some constant $B>0$ independent of $r_0, \e,$ and $R.$
\end{lemma}
\begin{proof}For our convenience, in this proof we drop the subscript $R$ and write $u$ instead of $\ur.$
Denote $\rho_1:=\max\limits_{x\in\p\Omega}|x|$ and let $a=\frac{5}{2}\rho_1.$ Then for any $x\in\mathbb{R}^n\setminus B_a(0)$, we have
\[\text{dist}(x, \p\Omega)+\rho_1>|x|,\] which implies
\[\text{dist}(x, \p\Omega)>|x|-\rho_1>\frac{|x|}{2}.\]
When $x_0\in\bar{B}_a(0)\setminus\Omega,$ we can see that $|x_0|^{-\a0-1}\geq a^{-\a0-1},$ and by Proposition \ref{C1-bound} we also know $|Du|$ has an uniform upper bound in $\bar{B}_a(0)\setminus\Omega.$ This gives when $x_0\in\bar{B}_a(0)\setminus\Omega,$ \eqref{c1inside} is satisfied for some constant $B=B(n, k, \p\Omega)>0.$

In the following, we will prove \eqref{c1inside} for the case when $x_0\in\mathbb{R}^n\setminus B_a(0).$
Choose any $x_0\in\mathbb{R}^n\setminus B_a(0)$ and denote $\tr=\frac{1}{8}|x_0|.$ Then
in $B_{\tr}(x_0)$ we have $\frac{7}{8}|x_0|\leq|x|\leq\frac{9}{8}|x_0|.$ Moreover, by Lemma \ref{C0-lem}, we also have
\[C_0\lt(\frac{9}{8}|x_0|\rt)^{-\a0}\leq|u(x)|\leq C_1\lt(\frac{7}{8}|x_0|\rt)^{-\a0}.\]
Here and in the rest of this paper we will always assume $\e>0$ is very small.
We will let $M=4C_1\lt(\frac{7}{8}|x_0|\rt)^{-\a0}.$  Following the argument of Chou-Wang \cite{CW}, we consider the test function
\[G(x)=u_\xi(x)\varphi(u)\rho(x),\]
where $\xi$ is some unit vector field, $\varphi(\tau)=(M-\tau)^{-\frac{1}{2}}$ and $\rho(x)=1-\frac{|x-x_0|^2}{\tr^2}.$ We will show
$|Du(x_0)|\leq C_2\frac{M}{\tr}.$
Suppose $G$ attains its maximum at $x=\hat{x}$ and $\xi=e_1.$ We rotate the coordinate so that $u_{\al\beta}=u_{\al\al}\delta_{\al\beta}$ for $\al, \beta\geq 2$. Then at $\hat{x}$, we have $G_i=0$ and
$\{G_{ij}\}\leq 0,$ $i, j\geq 1.$ That is
\begin{equation}
\label{c1i1}
u_{1i}=\frac{-u_1}{\varphi\rho}(u_i\varphi'\rho+\varphi\rho_i),
\end{equation}
and
\begin{equation}
\label{c1i2}
\begin{aligned}
0\geq&\varphi\rho(f_\e)_1-\varphi\rho S_k^{ij}u_{1i}\lt(-\frac{\varphi_j}{\varphi}-\frac{\rho_j}{\rho}\rt)\\
&+S_k^{ij}\lt(\varphi_{ij}-\frac{\varphi_i\varphi_j}{\varphi}\rt)u_1\rho
+S_k^{ij}\rho_{ij}u_1\varphi-\frac{1}{\rho}S_k^{ij}\rho_i\rho_ju_1\varphi\\
=&\varphi\rho(f_\e)_1+kf_\e u_1\rho\varphi'+u_1\rho\lt(\varphi''-\frac{2\varphi'^2}{\varphi}\rt)S_k^{ij}u_iu_j+u_1\varphi S_k^{ij}\rho_{ij}\\
&-u_1\varphi'S_k^{ij}(u_i\rho_j+u_j\rho_i)-\frac{2u_1\varphi}{\rho}S_k^{ij}\rho_i\rho_j.
\end{aligned}
\end{equation}
By our choice of $\varphi$ we have
\[\varphi'(u)=\frac{1}{2}(M-u)^{-\frac{3}{2}},\,\,\varphi''(u)=\frac{3}{4}(M-u)^{-\frac{5}{2}}.\]
Therefore, $\varphi''-\frac{2\varphi'^2}{\varphi}=\frac{1}{4}(M-u)^{-\frac{5}{2}}>\frac{1}{16}M^{-\frac{5}{2}}.$
Moreover, by virtue of \eqref{c2g1} we can see \eqref{c1i2}
implies
\begin{equation}
\label{c1i3}
\begin{aligned}
0\geq&-\frac{C\varphi\rho f_\e}{\tr}+\rho S_k^{11}u_1^3\frac{1}{16}M^{-\frac{5}{2}}-\frac{2u_1\varphi}{\tr^2}\sum_i S_k^{ii}\\
&-\lt[\varphi'u_1^2|D\rho|+\frac{\varphi u_1}{\rho}|D\rho|^2\rt]C(n)\sum_i S_k^{ii}.\\
\end{aligned}
\end{equation}
This yields
\begin{equation}
\label{c1i4}
\begin{aligned}
0\geq&\rho S_k^{11}u_1^3-\frac{2u_1M^{-\frac{1}{2}}}{\tr^2}16M^{\frac{5}{2}}\sum_i S_k^{ii}-\lt(M^{-\frac{3}{2}}\frac{u_1^2}{\tr}+\frac{4u_1M^{-\frac{1}{2}}}{\rho\tr^2}\rt)C(n)16M^{\frac{5}{2}}\sum_i S_k^{ii}\\
&-\frac{CM^{-\frac{1}{2}}f_\e}{\tr}16M^{\frac{5}{2}}.
\end{aligned}
\end{equation}
where $C$ is independent of $M$ and $\tr.$

 Now, we prove by contradiction and assume
$|Du(x_0)|>C_2\frac{M}{\tr}$ for some undetermined constant $C_2>0$. Then since $G(\hat{x})\geq G(x_0),$ we have
\[u_1\rho(\hat{x})>\frac{C_0}{C_1}\lt(\frac{7}{9}\rt)^{\al_0}\cdot C_2\frac{M}{\tr}.\]
Suppose $\l'=(u_{22}, \cdots, u_{nn}).$ Since $S_k=S_{k-1}(\l')u_{11}+S_k(\l')-\sum\limits_{i\geq 2}S_{k-2}(\l'|i)u_{1i}^2,$
we get $S_k^{11}=S_{k-1}(\l')$ and $S_k^{ii}=S_{k-2}(\l'|i)u_{11}+S_{k-1}(\l'|i),\,\,i\geq 2.$
It is clear that from  \eqref{c1i1} we have, $u_{11}\leq-\frac{\varphi'}{2\varphi}u_1^2<0.$ Therefore,
$\sum\limits_{i\geq 2}S_k^{ii}<(n-k)S_k^{11},$ which yields $S_k^{11}>\frac{1}{n-k+1}\sum_i S_k^{ii}.$
On the other hand, since $S_k^{\frac{1}{k}}$ is concave, we know $\frac{1}{k}f_\e^{\frac{1}{k}-1}\sum_i S_k^{ii}>\lt(C_n^k\rt)^{\frac{1}{k}}.$
Thus, $\sum_i S_k^{ii}>C(n, k)f_\e^{1-\frac{1}{k}}.$ Substituting the above inequalities into \eqref{c1i4}, we obtain
\begin{equation}
\label{c1i5}
\begin{aligned}
0&\geq\frac{1}{n-k+1}\rho u_1^3-\frac{32M^2}{\tr^2}u_1-16C(n)\frac{M}{\tr}u_1^2\\
&-64C(n)\frac{M^2}{\tr^2}\frac{u_1}{\rho}-\frac{CM^2}{\tr}f_\e^{\frac{1}{k}}.
\end{aligned}
\end{equation}
Note that $f_\e\sim\tr^{-(n+1)},$ it's easy to see that $f_\e^{\frac{1}{k}}<\frac{M}{\tr^2}.$ Therefore, when
$u_1\rho(\hat{x})>\frac{C_0}{C_1}\lt(\frac{7}{9}\rt)^{\al_0}\cdot C_2\frac{M}{\tr}$ and $C_2>C(n, k, \p\Omega)$ is sufficiently large, the right hand side of \eqref{c1i5}
would be positive. This leads to a contradiction.
\end{proof}

\begin{lemma}(\textbf{$C^2$ decay estimates})
\label{C2-interior-lem}
For any $x_0\in\mathbb{R}^n\setminus\Omega,$ denote $|x_0|=r_0.$ Let $\ur$ be a solution to equation \eqref{ed1*} with $R\gg r_0^{1+\frac{2}{\a0}}.$ Then we have
\begin{equation}\label{c2inside}
|D^2\ur(x_0)|\leq B|x_0|^{-\a0-2}\end{equation}
for some constant $B>0$ independent of $r_0, \e,$ and $R.$
\end{lemma}
\begin{proof}
For our convenience, in this proof we drop the subscript $R$ and write $u$ instead of $\ur.$
Let $u_0=u(x_0)$, then by Lemma \ref{C0-lem} we know $C_0r_0^{-\a0}<|u_0|<C_1r_0^{-\a0},$ where $\a0=\frac{n}{k}-2.$
We will consider the domain $\Omega_u:=\{x\in B_R\setminus \bar{\Omega}\mid |u|<2|u_0|\}.$

\textbf{Step1.} In this step, we want to show when $x\in \Omega_u,$ $|Du(x)|<c_1r_0^{-\a0-1}.$
By Lemma \ref{C0-lem} and Lemma \ref{C1-interior-lem} we know on $\p\Omega_u\setminus\p B_R$ we have
$|Du(x)|<c_2r_0^{-\a0-1}.$ By Lemma \ref{C1-outside-lem} we also know on $\p B_R$ $|Du(x)|<c_3R^{-\a0-1}.$
Applying Lemma \ref{C1-global-lem} in the domain $\Omega_u$ and choose $A=\frac{c_4}{r_0}$ we obtain
\[|Du|<c_5r_0^{-\a0-1}\,\,\mbox{in $\Omega_u.$}\]
Here $c_1,c_2,c_3,c_4,c_5$ are uniform constants independent of $r_0, \e,$ and $R$.

\textbf{Step2.} In this step we will prove \eqref{c2inside}. Denote $V:=|Du|^2$ and
$M:=4\max\limits_{x\in\bar{\Omega}_u}V\sim r_0^{-2(\a0+1)}.$ Let $\varphi(V)=(M-V)^{-\beta},$ where $\beta=\frac{\a0+2}{2\a0+2}.$ By our choice of $\beta,$ we can see that
$\varphi(V)\sim r_0^{\a0+2}.$

Consider $G=\rho^4u_{\xi\xi}\varphi(V)$, where $\rho=1-\frac{u}{2u_0}$ and $\xi$ is some unit vector field. Applying inequality \eqref{ed5*}, Lemma \ref{C2-outside-mix-lem}, and Lemma
\ref{C2-inside-normal-lem} we know when $R\gg r_0^{1+\frac{2}{\a0}},$ $G\leq CR^{-\a0}r_0^{\a0+2}\leq C$ on $\p B_R$. Moreover, on $\p\Omega_u\setminus\p B_R,$ we have $G=0.$ Our goal is to show
$\max\limits_{\bar{\Omega}_u}G\leq C,$ for some $C$ independent of $r_0, R,$ and $\e.$

Following the argument of Chou-Wang \cite{CW}, we assume $G_{\max}$ is achieved at an interior point $\hat{x}.$ We may also rotate the coordinate such that
at $\hat{x},$ we have $u_{ij}(\hat{x})=u_{ii}(\hat{x})\delta_{ij}=\lambda_i\delta_{ij}$ and $u_{11}\geq\cdots\geq u_{nn}.$ Then at this point, we have
\begin{equation}
\label{c2i1}
0=\frac{G_i}{G}=4\frac{\rho_i}{\rho}+\frac{\varphi_i}{\varphi}+\frac{u_{11i}}{u_{11},}
\end{equation}
and
\begin{equation}
\label{c2i2}
\begin{aligned}
0&\geq 4F^{ii}\lt(\frac{\rho_{ii}}{\rho}-\frac{\rho_i^2}{\rho^2}\rt)+F^{ii}\lt(\frac{\varphi_{ii}}{\varphi}-\frac{\varphi_i^2}{\varphi^2}\rt)
+F^{ii}\lt(\frac{u_{11ii}}{u_{11}}-\frac{u_{11i}^2}{u_{11}^2}\rt),
\end{aligned}
\end{equation}
where $F=S_k^{\frac{1}{k}}$ and $F^{ii}=\frac{\p F}{\p u_{ii}}.$
We will analyze \eqref{c2i2} in two cases.

Case 1. $u_{kk}\geq\theta_0u_{11}$ at $\hat{x},$ where $\theta_0>0$ is some fixed small constant to be determined.
By \eqref{c2i1} we get for any $\gamma>0$
\[\lt(\frac{u_{11i}}{u_{11}}\rt)^2\leq(1+\gamma)\frac{\varphi_i^2}{\varphi^2}+16\lt(1+\frac{1}{\gamma}\rt)\frac{\rho_i^2}{\rho^2}\]
Plugging it into \eqref{c2i2}, we obtain
\begin{equation}
\label{c2i3}
\begin{aligned}
0&\geq 4F^{ii}\lt\{\frac{\rho_{ii}}{\rho}-\lt[1+4\lt(1+\frac{1}{\gamma}\rt)\rt]\frac{\rho_i^2}{\rho^2}\rt\}+F^{ii}\lt[\frac{\varphi_{ii}}{\varphi}-(2+\gamma)\frac{\varphi_i^2}{\varphi^2}\rt]+F^{ii}\frac{u_{11ii}}{u_{11}},
\end{aligned}
\end{equation}
Note that
\[F^{ii}\lt[\frac{\varphi_{ii}}{\varphi}-(2+\gamma)\frac{\varphi_i^2}{\varphi^2}\rt]=
\lt[\frac{\varphi''}{\varphi}-(2+\gamma)\frac{\varphi'^2}{\varphi^2}\rt]F^{ii}V_i^2+\frac{\varphi'}{\varphi}F^{ii}V_{ii}.\]
By a straightforward calculation we can see that, in order for $\frac{\varphi''}{\varphi}-(2+\gamma)\frac{\varphi'^2}{\varphi^2}$ to be nonegative,
we need to choose $\gamma\leq\frac{\a0}{\a0+2}.$ In this case we have
\[F^{ii}\lt[\frac{\varphi_{ii}}{\varphi}-(2+\gamma)\frac{\varphi_i^2}{\varphi^2}\rt]
\geq\frac{\varphi'}{\varphi}F^{ii}(2u_{ii}^2+2u_ku_{kii}).\]
Denote $\hat{f}=f_\e^{\frac{1}{k}},$ then when  $\gamma\leq\frac{\a0}{\a0+2}$, \eqref{c2i3} becomes
\begin{equation}
\label{c2i4}
\begin{aligned}
0&\geq\frac{4\hat{f}}{2|u_0|\rho}-\frac{4[4(1+\gamma^{-1})+1]|Du|^2}{4\rho^2u_0^2}\sum_i F^{ii}+2\frac{\varphi'}{\varphi}F^{ii}u_{ii}^2+\frac{2\varphi'}{\varphi}Du\cdot D\hat{f}+\frac{\hat{f}_{11}}{u_{11}}
\end{aligned}
\end{equation}
By formula (3.2) in \cite{CW}, we get
\[\sum_{i=1}^n F^{ii}u_{ii}^2\geq F^{kk}u_{kk}^2\geq c_0u_{11}^2\sum_i F^{ii}\]
for some constant $c_0=c_0(n, k, \theta_0).$
Recall \eqref{c2g1} and \eqref{c2g2} we obtain
\begin{eqnarray}\label{CC1}
\hat{f}_i=\frac{1}{k}f_\e^{\frac{1}{k}-1}(f_\e)_i\sim \frac{\hat{f}}{r},
\end{eqnarray}
and
\begin{eqnarray}\label{CC2}
\hat{f}_{11}=\frac{1}{k}f_\e^{\frac{1}{k}-1}(f_\e)_{11}+\frac{1}{k}\lt(\frac{1}{k}-1\rt)
f_\e^{\frac{1}{k}-1}(f_\e)^2_1\sim\frac{\hat{f}}{r^2}.
\end{eqnarray}
Therefore, \eqref{c2i4} implies
\begin{equation}
\label{c2i5}
\begin{aligned}
0&\geq-\frac{C|Du|^2}{4u_0^2\rho^2}\sum F^{ii}+2\frac{\varphi'}{\varphi}c_0u_{11}^2\sum F^{ii}-\frac{C\varphi'}{\varphi}|Du|\frac{\hat{f}}{r}-\frac{C\hat{f}}{r^2u_{11}}.
\end{aligned}
\end{equation}
Since $\sum F^{ii}\geq(C_n^k)^{\frac{1}{k}},$ we get
\begin{equation}
\label{c2i6}
2\frac{\varphi'}{\varphi}c_0u_{11}^2<\frac{C|Du|^2}{4u_0^2\rho^2}+\frac{C\varphi'}{\varphi}|Du|\frac{\hat{f}}{r}+\frac{C\hat{f}}{r^2u_{11}}.
\end{equation}
Denote $u_{11}\rho(\hat{x}):=X$ then \eqref{c2i6} implies
\begin{equation}
\label{c2i7}
r_0^{2(\a0+1)}X^2\leq Cr_0^{-2}+Cr_0^{-2-\frac{1}{k}}+C\frac{r_0^{-\a0-4-\frac{1}{k}}}{X}.
\end{equation}
It's easy to see that if $X>Br_0^{-(\a0+2)}$ for some $B>0$ large, then inequality \eqref{c2i7} would not hold,
which leads to a contradiction.

Case 2. $u_{kk}<\theta_0u_{11}$ at $\hat{x}.$ By \eqref{c2i1} we have
\[\frac{u_{111}^2}{u_{11}^2}\leq(1+\gamma)\frac{\varphi_1^2}{\varphi^2}+16\lt(1+\gamma^{-1}\rt)\frac{\rho_1^2}{\rho^2},\] and
for $i\geq 2,$
\[4\frac{\rho_i^2}{\rho^2}=\frac{1}{4}\lt(\frac{\varphi_i}{\varphi}+\frac{u_{11i}}{u_{11}}\rt)^2
\leq\frac{1}{2}\lt(\frac{\varphi_i}{\varphi}\rt)^2+\frac{1}{2}\frac{u_{11i}^2}{u_{11}^2}.\]
Here, $\gamma$ takes the same value as in Case 1.
Equation \eqref{c2i2} can be written as
\begin{equation}
\label{c2i8}
\begin{aligned}
0&\geq\lt\{\sum_{i=1}^n\lt[4F^{ii}\frac{\rho_{ii}}{\rho}+F^{ii}\lt(\frac{\varphi_{ii}}{\varphi}-(2+\gamma)\frac{\varphi_i^2}{\varphi^2}\rt)\rt]\rt.\left.-4[1+4(1+\gamma^{-1})]F^{11}\frac{\rho_{1}^{2}}{\rho^2}\right\}\\
&+\lt\{\sum_{i=1}^nF^{ii}\frac{u_{11ii}}{u_{11}}-\frac{3}{2}\sum_{i=2}^nF^{ii}\frac{u_{11i}^2}{u_{11}^2}\rt\}\\
&:=I_1+I_2
\end{aligned}
\end{equation}
A straightforward calculation implies
\begin{equation}
\label{c2i9}
\begin{aligned}
I_1&\geq\frac{4\hat{f}}{2|u_0|\rho}+2\frac{\varphi'}{\varphi}F^{ii}u_{ii}^2+2\frac{\varphi'}{\varphi}Du\cdot D\hat{f}-C\frac{F^{11}|Du|^2}{4u_0^2\rho^2}.
\end{aligned}
\end{equation}
It's well known that $F$ is concave. Therefore,
\begin{equation}
\label{c2i10}
\begin{aligned}
&F^{pq,rs}u_{pq1}u_{rs1}\leq 2\sum\limits_{p>1}\frac{F^{pp}-F^{11}}{\l_p-\l_1}u^2_{11p}\\
&=\frac{2}{k}f_\e^{\frac{1}{k}-1}\sum\limits_{p>1}\frac{S_{k-1}(\l|p)-S_{k-1}(\l|1)}{\l_p-\l_1}u^2_{11p}\\
&=-\frac{2}{k}f_\e^{\frac{1}{k}-1}\sum\limits_{p>1}S_{k-2}(\l|1p)u^2_{11p}.
\end{aligned}
\end{equation}
This gives
\begin{equation}
\label{c2i11}
u_{11}I_2\geq\hat{f}_{11}+\frac{1}{k}f_\e^{\frac{1}{k}-1}\lt[2\sum_{p>1}S_{k-2}(\l|1p)-\frac{3S_{k-1}(\l|p)}{2\l_1}\rt]u_{11p}^2.
\end{equation}
Note that we assumed $u_{kk}<\theta_0u_{11}.$ Since $\sum_{j=k}^nu_{jj}\geq 0$ we can derive $|u_{jj}|\leq C(n, k)\theta_0u_{11},$ $j\geq k.$
Applying Lemma 3.1 of \cite{CW}, we know when $\theta_0$ is sufficiently small, we have
$2\l_1S_{k-2}(\l|1p)>\frac{3}{2}S_{k-1}(\l|p).$ Substituting \eqref{c2i9} and \eqref{c2i11} into \eqref{c2i8}, we btain
\begin{equation}
\label{c2i12}
\begin{aligned}
0&\geq\frac{2\hat{f}}{|u_0|\rho}+2\frac{\varphi'}{\varphi}F^{ii}u_{ii}^2+2\frac{\varphi'}{\varphi}Du\cdot D\hat{f}-C\frac{F^{11}|Du|^2}{4u_0^2\rho^2}+\frac{\hat{f}_{11}}{u_{11}}.
\end{aligned}
\end{equation}
Combining with \eqref{CC1} and \eqref{CC2}, we get at $\hat{x}$
\begin{align*}
0&\geq\frac{2\hat{f}}{|u_0|\rho}+\frac{2\varphi'}{\varphi}F^{ii}u_{ii}^2-\frac{C\varphi'}{\varphi}r_0^{-\a0-2}\hat{f}-\frac{CF^{11}r_0^{-2}}{\rho^2}-\frac{C\hat{f}}{u_{11}r_0^2}.
\end{align*}
Since $\frac{\varphi'}{\varphi}=\beta(M-V)^{-1}\geq\beta M^{-1}\sim r_0^{2(\a0+1)}$ and when $G(\hat{x})$ is very large, we have
$u_{11}\rho(\hat{x})>Cr_0^{-\a0-2},$ which implies
$\frac{\varphi'}{\varphi}F^{ii}u_{ii}^2>\frac{CF^{11}r_0^{-2}}{\rho^2}.$ Therefore, in this case, \eqref{c2i12} becomes
\begin{equation}
\label{c2i13}
0\geq\frac{b\hat{f}}{2|u_0|\rho}+\frac{\varphi'}{\varphi}F^{ii}u_{ii}^2
-C\frac{\varphi'}{\varphi}r_0^{-\a0-2}\hat{f}-\frac{C\hat{f}}{u_{11}r_0^2}.
\end{equation}
It is easy to see that $\sum F^{ii}u_{ii}^2>c_0\hat{f}\l_1.$ Same as before assume
$u_{11}\rho(\hat{x}):=X> Br_0^{-\a0-2}$ for some sufficiently large constant $B>0$, then by \eqref{c2i13},
we get
\[r_0^{2(\a0+1)}Br_0^{-\a0-2}\leq Cr_0^{2(\a0+1)-\a0-2}+\frac{Cr_0^{-2}}{X}\rho^2.\]
When $B>0$ is very large, this leads to a contradiction. Therefore, we have proved Lemma \ref{C2-interior-lem}.
\end{proof}
\subsection{Convergence}
\label{subsection-conv} In this subsection, we will show that there exists a subsequence of $\{u_R\}$ converging to the desired solution $u$ of \eqref{k-hessian}.

In the following , for any fixed $R>R_0$ very large, let $\e_R=c_0R^{-k(\al_0+3)},$ where $c_0=c_0(\p\Omega, n, k)>0$ is a fixed small constant.
Let $u^{\e_R}_R$ be the solution of \eqref{ed1*} with $\e=\e_R,$ then we have the following lemma.
\begin{lemma}
\label{subsection conv-lem1}
For any $R>R_0$ very large, let
\[\Psi_R=|D\uer|-b_0(-\uer)^A-\theta_0\frac{|x|^2}{R^{\al_0+3}},\] where $A=\frac{\al_0+1}{\al_0},$ and
$b_0=b_0(\p\Omega, n, k),$ $\theta_0=\theta_0(\p\Omega, n, k)>0$ small such that $\Psi_R>0$ on $\p(B_R\setminus\bar\Omega).$ Then
\[\Psi_R>0\,\,\mbox{in $B_R\setminus\bar\Omega.$}\]
\end{lemma}
\begin{proof}
Before proving this lemma, we want to point out that by Lemma \ref{C1-inside-lower-lem}, and \ref{C1-outside-lem}, we know there always exist $b_0=b_0(\p\Omega, n, k),$ $\theta_0=\theta_0(\p\Omega, n, k)>0$ small such that $\Psi_R>0$ on $\p(B_R\setminus\bar\Omega).$

Now, we will prove this lemma by a contradiction argument. We assume $\Psi$ achinves its non-positive minimum point at $x_0.$ Then at this point we may rotate the coordinate such that
$|D\uer|=(\uer)_1.$ A straightforward calculation yields at $x_0$
\[S_k^{ij}(\Psi_R)_{ij}\leq(f_{\e_R})_1+Ab_0|\uer|^{A-1}kf_{\e_R}-2\frac{\theta_0}{R^{\al_0+3}}\sum S_k^{ii}.\]
Since $S_k^{1/k}$ is concave, we get $\sum S_k^{ii}>c_1f_{\e_R}^{1-\frac{1}{k}}$ for some $c_1=c_1(n, k)>0$ only depending on $n, k.$
By a proper choice of $c_0$ that is independent of $x_0$ and $R,$ we can see that $S_k^{ij}(\Psi_R)_{ij}<0$ at $x_0.$ This leads to a contradiction.
\end{proof}

Combining Lemma \ref{C0-interior-lem}, \ref{C1-interior-lem}, \ref{C2-interior-lem}, and \ref{subsection conv-lem1} with the standard convergence theorem we conclude
\begin{theorem}
\label{existence-theorem}
There exists a $k$-admissible solution $u$ of equation \eqref{k-hessian} satisfying
\[|u(x)|<B|x|^{-\al_0},\,\,|Du(x)|<B|x|^{-\al_0-1},\,\,\mbox{and $|D^2u(x)|<B|x|^{-\al_0-2},$}\]
for any $x\in \R^n\setminus\Omega.$
Here $B>0$ is a constant depending on $n, k, $ and $\Omega.$ Moreover, we also have
\[|Du|-b_0|u|^{\frac{\al_0+1}{\al_0}}\geq 0\,\,\mbox{in $\R^n\setminus\Omega$}\]
for some $b_0=b_0(\p\Omega)>0.$
\end{theorem}

\section{Asymptotic behavior at $\infty$}
\label{section ab}
\label{asymp}
In this section, we will consider the asymptotic behavior of the solution $u$ of equation \eqref{k-hessian}.
By using the maximum principle, we obtain the lemma below.
\begin{lemma}
\label{ab-max-lem}
Let $G$ be a connected, bounded, open subset in $\mathbb{R}^n\setminus\{0\}$ and $\mu=-|x|^{2-\frac{n}{k}}.$
Assume $u$ is a $k$-admissible solution satisfies $S_k(\n^2 u)=0$. If $\frac{u}{\mu}$ achieves its maximum (minimum) in $G$, then
$\frac{u}{\mu}$ is a constant.
\end{lemma}

\begin{lemma}
\label{asymp-lem}
Let $u$ be a $k$-admissible solution of \eqref{k-hessian}. Then there exists some constant $C_0\leq\gamma\leq C_1$ such that
$u\goto\gamma\mu$ in $C^2$ topology as $|x|\goto\infty.$
\end{lemma}
\begin{proof}This proof follows the idea of \cite{KV}.
Without loss of generality, in the following, we assume $\p\Omega\subset B_1(0).$ By Lemma
 \ref{C0-lem}, we can define $\gamma^+$ and $\gamma^-$ by
\[\gamma^+=\limsup\limits_{x\goto\infty}\frac{u(x)}{\mu(x)},\,\,
\gamma^-=\liminf\limits_{x\goto\infty}\frac{u(x)}{\mu(x)}.\]
It is clear that $C_0\leq \gamma^+,\gamma^-\leq C_1$. We can also assume $\max\limits_{x\in\mathbb{S}^{n-1}}|u(x)|=1.$ Otherwise, we use  $\frac{u(x)}{\max\limits_{x\in\mathbb{S}^{n-1}}u(x)}$ to replace $u(x)$.
If $\gamma^+=\gamma^-=1$ we would have $\lim\limits_{x\goto\infty}\frac{u(x)}{\mu(x)}=1.$
Thus we will assume $\gamma^+>1$ (or similarly $\gamma^-<1$).
Let $\tilde{\gamma}(r)=\sup\limits_{1\leq|x|\leq r}\frac{u(x)}{\mu(x)},$ then it's easy to see that $\tilde{\gamma}(r)\geq 1$ is nondecreasing.
Therefore we have $\lim\limits_{r\goto\infty}\tilde{\gamma}(r)=\gamma^+.$

Define $u_r$ on $\Lambda_r=\{\xi\mid|\xi|>\frac{1}{r}\}$ by $u_r(\xi)=-\frac{u(r\xi)}{\mu(r)}.$ By Theorem \ref{existence-theorem} we have
\[|u_r(\xi)|\leq C|\xi|^{-\a0},\]
\[|D u_r(\xi)|\leq C|\xi|^{-\a0-1},\]
and
\[|D^2u_r(\xi)|\leq C|\xi|^{-\a0-2}.\]
This implies there exists a function $v$ satisfying $S_k(D^2 v)=0$ such that
$u_r(\xi)\goto v(\xi)$ in $C^2$ topology on any compact subset of $\mathbb{R}^n\setminus\{0\}.$ Moreover, we have
\[\frac{u_r(\xi)}{\mu(\xi)}=-\frac{u(r\xi)}{\mu(r\xi)}\cdot\frac{\mu(r\xi)}{\mu(\xi)\mu(r)}
=\frac{u(r\xi)}{\mu(r\xi)}.\]
Suppose $\tilde{\gamma}(r)=\frac{u(x_r)}{\mu(x_r)},$ by the monotonicity of $\tilde{\gamma},$
we know $|x_r|=r.$ Let $\xi_r=\frac{x_r}{r},$ then we have $\frac{u_r(\xi_r)}{\mu(\xi_r)}=\tilde{\gamma}(r).$
Choose a subsequence of $\{\xi_r\},$ denote by $\{\xi_{r_n}\},$ and assume $\{\xi_{r_n}\}\goto\xi_0.$ Then we have
\[\frac{v(\xi_0)}{\mu(\xi_0)}=\gamma^+\,\,\mbox{and $\frac{v(\xi)}{\mu(\xi)}\leq\gamma^+$ for any $\xi\in\mathbb{R}^n\setminus\{0\}.$}\]
By Lemma \ref{ab-max-lem} we know $v(\xi)=\gamma^+\mu(\xi)=\lim\limits_{r\goto\infty}u_r(\xi)$ uniformly on every compact subset of
$\mathbb{R}^n\setminus\{0\}$ in $C^2$ topology.
This gives $\lim\limits_{|x|\goto\infty}\frac{u(x)}{\mu(x)}=\gamma$ and Lemma \ref{asymp-lem} is proved.
\end{proof}

\section{The monotonicity quantity}
\label{section mq}
In this section, we will derive a monotonicity quantity satisfied by the solution $u$ of \eqref{k-hessian}.

To start, let us consider the radial function $$\mu=-|x|^{2-\frac{n}{k}}.$$
One can verify that for $n>2k$, $\lambda[\n^2\mu]$ belongs to $\bar \Gamma_k^+$ and $\mu$ satisfies the Hessian equation
$$S_k(\n^2 u)=0.$$
Here $\lambda[\n^2\mu]=(\lambda_1, \cdots, \lambda_n)$ are the eigenvalues of the Hessian of $\mu.$
In polar coordinates, a straightforward calculation gives  $$\n_r \mu= (\frac{n}{k}-2)|x|^{1-\frac{n}{k}}, \quad \n_\theta \mu=0,$$
$$\quad  \n^2 \mu(\p_r, \p_r)=-(\frac{n}{k}-2)(\frac{n}{k}-1)|x|^{-\frac{n}{k}}, \quad \n^2 \mu(\p_r, \p_\theta)=0,\quad g^{\theta\theta}\n^2 \mu(\p_\theta, \p_\theta)= (\frac{n}{k}-2)|x|^{-\frac{n}{k}},$$
$$\Delta \mu= (\frac{n}{k}-2)(n-\frac{n}{k})|x|^{-\frac{n}{k}}.$$
One sees on a level set of $\mu$,
\be\label{mq1.1*}\frac{1}{|\n \mu|}S_k^{ij}\mu_i\mu_j=S_k^{rr}\mu_r=\binom{n-1}{k-1}  (\frac{n}{k}-2)^{k} |x|^{1-n}.\ee
\be\label{mq1.2*}\frac{|\n \mu|}{(-\mu)^{\frac{1-\frac{n}{k}}{2-\frac{n}{k}}}}=\frac{n}{k}-2.\ee
Thus the quantity $\Phi: (-\infty, -1]\to \rr$,
\be\label{def-phi}
\Phi(\tau)=\int_{\{u=1/\tau\}}  \frac{1}{|\n u|}S_k^{ij}u_iu_j\left(\frac{|\n u|}{(-u)^{\frac{k-n}{2k-n}}}\right)^{\beta}  dx
 \ee
 is a constant independent of $\tau$ for the radial solution $\mu$ with any constant $\beta\in\R$.

We show in the following that for the $k$-admissible solution $u$ of \eqref{k-hessian}, when $\beta\ge \frac{n-2k}{n-k},$ $\Phi$ is monotone.
\begin{proposition}
\label{monotone-quantity}
Let $n>2k$ and $\beta\ge \frac{n-2k}{n-k}$. Let $u$ be a $k$-admissible solution of
\begin{equation}\label{eqn:1.0}
\left\{
\begin{aligned}
&S_k(\n^2 u)=0\ \ {\rm in}\ \ \mathbb R^n\setminus\bar{\Omega}\\
&u=-1\ \ {\rm on} \ \ \partial\Omega\\
&u(x)\rightarrow 0\ \ {\rm as}\ \ |x|\rightarrow\infty.
\end{aligned}\right.
\end{equation}
 Then $\Phi(\tau)$ defined by \eqref{def-phi} is monotone non-decreasing for regular value $\frac{1}{\tau}$, i.e., for $\tau$ such that $\{u=\frac{1}{\tau}\}$ is a regular level set.
\end{proposition}

\begin{proof}
In this proof, we denote $l=\frac{n-k}{n-2k}>0$ and we also note that $\p_j(S_k^{ij}u_i)=0$. Using the Divergence theorem and the co-area formula we have
\begin{eqnarray*}
&&\Phi(\tau)-\Phi(-\infty)\nonumber\\
&=&\int_{\{u=1/\tau\}}  \frac{1}{|\n u|}S_k^{ij}u_iu_j\left(\frac{|\n u|}{(-u)^{\frac{k-n}{2k-n}}}\right)^{\beta}  dx-\int_{\{u=0\}}  \frac{1}{|\n u|}S_k^{ij}u_iu_j\left(\frac{|\n u|}{(-u)^{\frac{k-n}{2k-n}}}\right)^{\beta}  dx \nonumber\\
&=&-\int_{\{u>1/\tau\}}\p_j\left[(-u)^{-l\beta} |\n u|^\beta S_k^{ij}u_i\right]\nonumber\\
&=&- \int_{\{u>1/\tau\}}\p_j\left[(-u)^{-l\beta} |\n u|^\beta\right] S_k^{ij}u_i \nonumber\\
&=&- \int_{1/\tau}^0\int_{\{u=s\}}\p_j\left[(-u)^{-l\beta} |\n u|^\beta\right] S_k^{ij}u_i\frac{1}{|\n u|}.
\end{eqnarray*}
Taking the derivative of $\Phi$, we get
\begin{eqnarray}
\Phi'(\tau)&=&-\int_{\{u=1/\tau\}}u^2\p_j\left[(-u)^{-l\beta} |\n u|^\beta \right] S_k^{ij}u_i\frac{1}{|\n u|}\nonumber\\
&=&\int_{\{u>1/\tau\}}\p_i\left\{u^2\p_j\left[(-u)^{-l\beta} |\n u|^\beta \right] S_k^{ij}\right\}\nonumber\\
&=&\int_{\{u>1/\tau\}}u^2 S_k^{ij}\p_i\p_j\left[(-u)^{-l\beta} |\n u|^\beta \right]+2u\p_j\left[(-u)^{-l\beta} |\n u|^\beta \right] S_k^{ij}u_i,\nonumber\\
\label{eq-x1}
\end{eqnarray}
where we have used $\p_jS_k^{ij}=0.$
Since $S_k(\n^2 u)=0$, we have
\begin{eqnarray*}
S_k^{ij}u_{ijk}=0,
\end{eqnarray*}
which yields
\begin{eqnarray*}
S_k^{ij}(|\n u|^\beta)_{ij}=\beta|\n u|^{\beta-2}S_k^{ij}u_{ik}u_{jk}+\beta (\beta-2)|\n u|^{\beta-2}S_k^{ij}\p_i |\n u|\p_j |\n u|.
\end{eqnarray*}
Therefore, we deduce
\begin{eqnarray*}
&&u^2S_k^{ij}\p_i\p_j\lt[(-u)^{-l\beta} |\n u|^\beta \rt]+2u\p_j\lt[(-u)^{-l\beta} |\n u|^\beta \rt] S_k^{ij}u_i\nonumber\\
&=&(-l\beta)(-l\beta-1)(-u)^{-l\beta}S_k^{ij}u_iu_j |\n u|^\beta+ 2(l\beta)\beta (-u)^{-l\beta+1}S_k^{ij}u_i |\n u|^{\beta-1}\p_j|\n u|\nonumber\\
&&+ \beta (-u)^{-l\beta+2}|\n u|^{\beta-2}\Big[S_k^{ij}u_{ik}u_{jk}+(\beta-2)S_k^{ij}\p_i|\n u|\p_j|\n u|\Big]\nonumber\\
&&-2\beta(-u)^{-l\beta+1} |\n u|^{\beta-1}S_k^{ij}u_i \p_j|\n u|-2l\beta (-u)^{-l\beta}S_k^{ij}u_iu_j |\n u|^{\beta}\nonumber\\
&=&\beta (-u)^{-l\beta+2}|\n u|^{\beta-2}\lt\{S_k^{ij}u_{ik}u_{jk}+\left[(\frac{1}{l}-2)+\frac{1}{l}(l\beta-1)\right]S_k^{ij}\p_i|\n u|\p_j|\n u|\rt\}\nonumber\\
&&+2\beta (l\beta-1)  (-u)^{-l\beta+1} |\n u|^{\beta-1} S_k^{ij}u_i \p_j|\n u| +l\beta(l\beta-1)(-u)^{-l\beta}|\n u|^\beta S_k^{ij}u_iu_j\nonumber\\
&=&\beta (-u)^{-l\beta+2}|\n u|^{\beta-2}\left[S_k^{ij}u_{ik}u_{jk}+(\frac{1}{l}-2)S_k^{ij}\p_i|\n u|\p_j|\n u|\right]\nonumber\\
&&+ \frac{1}{l}\beta (l\beta-1)(-u)^{-l\beta+2}|\n u|^{\beta-2} S_k^{ij}\left(\p_i|\n u| -l \frac{|\n u|}{u}u_i\right)\left(\p_j|\n u| -l \frac{|\n u|}{u}u_j\right).
\end{eqnarray*}
By our choice of $l$ we can see that $\frac{1}{l}-2=-\frac{n}{n-k}$ and $\beta\ge \frac{1}{l}$, applying the Kato's inequality (see Proposition \ref{Kato}), we obtain that
\begin{eqnarray*}
&&u^2S_k^{ij}\p_i\p_j\lt[(-u)^{-l\beta} |\n u|^\beta \rt]+2u\p_j\lt[(-u)^{-l\beta} |\n u|^\beta \rt] S_k^{ij}u_i\ge 0.
\end{eqnarray*}
In view of \eqref{eq-x1}, we have shown that
$\Phi'(\tau)\ge 0$. This completes the proof of Proposition \ref{monotone-quantity}.
\end{proof}
In view of Proposition \ref{monotone-quantity} we obtain
\begin{corollary}
\label{cor monotone-quantity}
Let $n>2k$ and $\beta\ge \frac{n-2k}{n-k}$. Let $u$ be a $k$-admissible solution of \eqref{k-hessian}
constructed in Subsection \ref{subsection-conv}.
 Then $\Phi(\tau)$ defined by \eqref{def-phi} satisfies
 \[\Phi(-1)\geq\Phi(-\infty).\]
\end{corollary}

Below, we prove the Kato's inequality.
\begin{proposition}[Kato's inequality]\label{Kato} Let $u$ be a $k$-admissible function satisfying $S_k(D^2 u)=0.$ Then at any regular point, i.e.,$|\n u|\neq 0$,  we have
\begin{equation}
\label{eqn:1.0*}
 \sum_mS_{k}^{ij}u_{im}u_{mj}- \frac{n}{n-k}S_k^{ij}D_i|D u|D_j|D u| \ge 0.
\end{equation}
\end{proposition}
\begin{proof}
In this proof, the index range are $1\leq\alpha,\beta,\ldots\leq n-1$ and $1\leq 1,j,m,\ldots\leq n$.
At any point $p$ satisfying $\n u(p)\neq 0$, choose $\{e_\alpha\}$ such that $$u_{\alpha\beta}(p)=\lambda_{\alpha}\delta_{\alpha\beta}, \quad e_n=\frac{D u}{|\n u|}(p),$$
and denote $\lambda'=(\lambda_1,\cdots,\lambda_{n-1})$. Note also that in the chosen coordinate we have $\n_i|\n u|=u_{in}$. Therefore at this point we have
\begin{equation}\label{eqn:1.5}
0=S_k(D^2u)=S_{k-1}(\lambda')u_{nn}+S_k(\lambda')-\sum_{\alpha}S_{k-2}(\lambda'|\alpha)u_{\alpha n}^2.
\end{equation}
Similarly, we can derive
\begin{equation}\label{eqn:1.6}
\begin{aligned}
S_{k+1}[u]
=&S_{k}(\lambda')u_{nn}+S_{k+1}(\lambda')-\sum_{\alpha}S_{k-1}(\lambda'|\alpha)u_{\alpha n}^2.
\end{aligned}
\end{equation}
A direct calculation shows that
\begin{equation}\label{eqn:1.1}
\begin{aligned}
  S_k^{\alpha\alpha}=
  %&\frac{1}{(k-1)!}\sum_{\substack{j_1\dots j_{k-1}\not=\alpha\\i_1\dots i_{k-1}\not=\alpha}}\delta^{j_1\dots j_{k-1}}_{i_1\dots i_{k-1}}u_{i_1j_1}\cdots u_{i_{k-1}j_{k-1}}\\
%&  =&\frac{1}{(k-1)!}(k-1)\sum_{\substack{j_2\dots j_{k-1}\not=\alpha,n\\i_2\dots i_{k-1}\not=\alpha,n}}\delta^{j_2\dots j_{k-1}}_{i_2\dots i_{k-1}}u_{i_2j_2}\cdots u_{i_{k-1}j_{k-1}}u_{nn}\\
%  &+\frac{(k-1)(k-2)(-1)}{(k-1)!}\sum_{\substack{j_1j_3\dots j_{k-1}\not=\alpha,n\\i_2i_3\dots i_{k-1}\not=\alpha,n}}\delta^{j_1j_3\dots j_{k-1}}_{i_2i_3\dots i_{k-1}}u_{j_1n}u_{i_2n}u_{i_3j_3}\cdots u_{i_{k-1}j_{k-1}}\\
%  &+\frac{1}{(k-1)!}\sum_{\substack{j_1\dots j_{k-1}\not=\alpha,n\\i_1\dots i_{k-1}\not=\alpha,n}}\delta^{j_1\dots j_{k-1}}_{i_1\dots i_{k-1}}u_{i_1j_1}\cdots u_{i_{k-1}j_{k-1}}\\
  S_{k-2}(\lambda'|\alpha)u_{nn}+S_{k-1}(\lambda'|\alpha)-\sum_{\beta\not=\alpha}S_{k-3}(\lambda'|\alpha\beta)u_{\beta n}^2
  \end{aligned},
\end{equation}
and
\begin{equation}\label{eqn:1.2}
  \begin{aligned}
  S_k^{\alpha n}
  %=&\frac{1}{(k-1)!}\sum_{\substack{ j_1\dots j_{k-1}\\i_1\dots i_{k-1}}}\delta^{\alpha j_1\dots j_{k-1}}_{ni_1\dots i_{k-1}}u_{i_1j_1}\cdots u_{i_{k-1}j_{k-1}}\\
  %=&\frac{1}{(k-1)!}(k-1)(-1)\sum_{\substack{j_2\dots j_{k-1}\not=\alpha,n\\i_2\dots i_{k-1}\not=\alpha,n}}\delta^{j_2\dots j_{k-1}}_{i_2\dots i_{k-1}}u_{i_2j_2}\cdots u_{i_{k-1}j_{k-1}}u_{\alpha n}\\
 % &+\frac{(k-1)(k-2)}{(k-1)!}\sum_{\substack{j_2j_3\dots j_{k-1}\not=\alpha,n\\i_1i_3\dots i_{k-1}\not=\alpha,n}}\delta^{j_1j_3\dots j_{k-1}}_{i_2i_3\dots i_{k-1}}u_{i_1n}u_{j_2\alpha}u_{i_3j_3}\cdots u_{i_{k-1}j_{k-1}}\\
=-S_{k-2}(\lambda'|\alpha)u_{\alpha n}.
  \end{aligned}
\end{equation}
It is easy to see that for $\alpha\not=\beta$, $S_k^{\alpha \beta}=0$.
%
%\begin{equation}\label{eqn:1.3}
%  \begin{aligned}
%  S_k^{\alpha \beta}=&\frac{1}{(k-1)!}\sum_{\substack{ j_1\dots j_{k-1}\\i_1\dots i_{k-1}}}\delta^{\alpha j_1\dots j_{k-1}}_{\beta i_1\dots i_{k-1}}u_{i_1j_1}\cdots u_{i_{k-1}j_{k-1}}\\
%  =&\frac{1}{(k-1)!}(k-1)(-1)\sum_{\substack{j_2\dots j_{k-1}\not=\alpha,\beta\\i_2\dots i_{k-1}\not=\alpha,\beta}}\delta^{j_2\dots j_{k-1}}_{i_2\dots i_{k-1}}u_{i_1j_1}\cdots u_{i_{k-1}j_{k-1}}u_{\alpha \beta}\\
%  &{\color{red}+\frac{(k-1)(k-2)(-1)}{(k-1)!}\sum_{\substack{j_2j_3\dots j_{k-1}\not=\alpha,\beta\\i_1i_3\dots i_{k-1}\not=\alpha,\beta}}\delta^{j_1j_3\dots j_{k-1}}_{i_2i_3\dots i_{k-1}}u_{i_1\beta}u_{j_2\alpha}u_{i_3j_3}\cdots u_{i_{k-1}j_{k-1}}}\\
%=&0.
%  \end{aligned}
%\end{equation}
%
Notice that
\begin{eqnarray}
&&  kS_k(\lambda')=\sum_{\alpha}S_{k-1}(\lambda'|\alpha)\lambda_{\alpha},\label{eqn:1.4} \\
  &&  kS_k(\lambda'|\beta)=\sum_{\alpha}S_{k-1}(\lambda'|\alpha\beta)\lambda_{\alpha}.\label{eqn:1.4'}
\end{eqnarray}

Next, computing $S_k^{ij}u_{in}u_{jn}$ we obtain
\begin{equation}\label{eqn:1.7}
\begin{aligned}
S_k^{ij}u_{in}u_{jn}=&S_{k}^{\alpha\alpha}u_{\alpha n}^2+2S_{k}^{\alpha n}u_{\alpha n}u_{nn}+S_k^{nn}u_{nn}^2\\
=&\sum_{\alpha}\big(S_{k-2}(\lambda'|\alpha)u_{nn}+S_{k-1}(\lambda'|\alpha)-\sum_{\beta\not=\alpha}
S_{k-3}(\lambda'|\alpha\beta)u_{\beta n}^2\big)u_{\alpha n}^2\\
&-2\sum_{\alpha}S_{k-2}(\lambda'|\alpha)u_{\alpha n}^2u_{nn}+S_{k-1}(\lambda')u_{nn}^2\\
=&S_{k-1}(\lambda')u_{nn}^2-\sum_{\alpha}S_{k-2}(\lambda'|\alpha)u_{\alpha n}^2u_{nn}\\
&+\sum_{\alpha}\Big(S_{k-1}(\lambda'|\alpha)-\sum_{\beta\not=\alpha}S_{k-3}(\lambda'|\alpha\beta)u_{\beta n}^2\Big)u_{\alpha n}^2
\\ \leq &-S_{k}(\lambda')u_{nn}+\sum_{\alpha}S_{k-1}(\lambda'|\alpha)u_{\alpha n}^2.
\end{aligned}\end{equation}
In the last inequality, we have used \eqref{eqn:1.5} and $S_{k-3}(\lambda'|\alpha\beta)\geq 0$.

We also have
\begin{equation}\label{eqn:1.80}
\begin{aligned}
S_{k}^{ij}u_{ik}u_{kj}&= S_1(D^2u)S_{k}(D^2u)-(k+1)S_{k+1}(D^2u)= -(k+1)S_{k+1}(D^2u).
\end{aligned}
\end{equation}
Since $\lambda[D^2u]\in\bar{\Gamma}_k,$ it is easy to deduce that $\lambda'\in\bar{\Gamma}_{k-1}.$

\textbf{Case1}: When at the point $p$ we have $S_{k-1}(\l')=0,$ by \eqref{eqn:1.5} we can see
\[S_k(\l')=\sum\limits_{\al}S_{k-2}(\l'|\al)u^2_{\al n}\geq 0.\]
This yields $S_k(\l')=0$ and for any $\al$ we have
\[S_{k-2}(\l'|\al)u^2_{\al n}=0,\] which in turn implies for any $\al$
\[S_{k-1}(\l'|\al)u^2_{\al n}=0.\]
Combining the above equalities with \eqref{eqn:1.7} we conclude, in this case
\[S_k^{ij}u_{in}u_{jn}\leq 0.\]
Therefore, we have
\[ S_{k}^{ij}u_{ik}u_{kj}- \frac{n}{n-k}S_k^{ij}u_{in}u_{jn}\geq 0.\]

\textbf{Case 2}: In the following we will assume $S_{k-1}(\l')>0$ at the point under consideration.
Applying\eqref{eqn:1.5}-\eqref{eqn:1.80}, we get
\begin{equation}\label{eqn:1.800}
\begin{aligned}
& S_{k}^{ij}u_{ik}u_{kj}- \frac{n}{n-k}S_k^{ij}u_{in}u_{jn}
\\=& -(k+1)S_{k+1}(D^2u)-\frac{n}{n-k}S_k^{ij}u_{in}u_{jn}\\
\geq & -(k+1)\Big[S_{k}(\lambda')u_{nn}+S_{k+1}(\lambda')-\sum_{\alpha}S_{k-1}(\lambda'|\alpha)u_{\alpha n}^2\Big] \\
&-\frac{n}{n-k}\Big[-S_{k}(\lambda')u_{nn}+\sum_{\alpha}S_{k-1}(\lambda'|\alpha)u_{\alpha n}^2\Big]\\
=&-(k+1)S_{k+1}(\lambda')-\frac{k(n-k-1)}{n-k}\Big[S_k(\lambda')u_{nn} -\sum_{\alpha}S_{k-1}(\lambda'|\alpha)u_{\alpha n}^2\Big]\\
=&-(k+1)S_{k+1}(\lambda')\\
&-\frac{k(n-k-1)}{n-k}\lt[-S_k(\lambda')\frac{S_k(\lambda')-\sum_{\alpha}S_{k-2}(\lambda'|\alpha)u_{\alpha n}^2}{S_{k-1}(\lambda')} -\sum_{\alpha}S_{k-1}(\lambda'|\alpha)u_{\alpha n}^2\rt]\\
=&-(k+1)S_{k+1}(\lambda')+\left(\frac{k(n-k-1)}{n-k}\right)\frac{S_k^2(\lambda')}{S_{k-1}(\lambda')}\\
&-\frac{k(n-k-1)}{n-k}\sum_{\alpha}\left(S_k(\lambda')\frac{S_{k-2}(\lambda'|\alpha)}{S_{k-1}(\lambda')}-S_{k-1}(\lambda'|\alpha)\right)u_{\alpha n}^2.
\end{aligned}
\end{equation}
Finally, the Newton-Maclaurin inequality implies
\begin{equation}\label{eqn:1.9}
\left(\frac{k(n-k-1)}{n-k}\right)\frac{S_k^2(\lambda')}{S_{k-1}(\lambda')}- (k+1)S_{k+1}(\lambda')\ge 0,
\end{equation}
and
\begin{equation}\label{eqn:1.10}
  S_{k-1}(\lambda'|\alpha)- \frac{S_k(\lambda')S_{k-2}(\lambda'|\alpha)}{S_{k-1}(\lambda')}=\frac{S^2_{k-1}(\lambda'|\alpha)-S_k(\lambda'|\alpha)S_{k-2}(\lambda'|\alpha)}{S_{k-1}(\lambda')}\ge 0.
\end{equation}
Inserting \eqref{eqn:1.9} and  \eqref{eqn:1.10} into  \eqref{eqn:1.800}, we conclude that
\begin{eqnarray*}
 S_{k}^{ij}u_{ik}u_{kj}- \frac{n}{n-k}S_k^{ij}u_{in}u_{jn}\ge 0.
\end{eqnarray*}
Thus we finish the proof.
\end{proof}
Finally, applying Lemma 2.2 in \cite{GS}, Theorem \ref{int-thm2}, and Corollary \ref{cor monotone-quantity}, Corollary \ref{int-cor1} follows directly.

\end{document}